\documentclass[a4paper,11pt]{amsart}
\usepackage{amssymb}
\usepackage{amscd}
\usepackage{comment}
\usepackage{amsmath,amsthm}
\usepackage[colorlinks=true]{hyperref}
\usepackage{booktabs,multirow}
\usepackage{tikz}
\usepackage{rotating}
\usetikzlibrary{patterns}
\usetikzlibrary{decorations.pathreplacing}
\usetikzlibrary{calc,through}

\allowdisplaybreaks[1]
\numberwithin{figure}{section}
\numberwithin{equation}{section}

\title[Generalized Dellac configurations]
{Generalized Dellac configurations}
\author[K.~Shigechi]{Keiichi~Shigechi}
\email{k1.shigechi AT gmail.com}
\date{\today}

\newcommand\tikzpic[2]{
\raisebox{#1\totalheight}{
\begin{tikzpicture}
#2
\end{tikzpicture}
}}
\newtheorem{theorem}[figure]{Theorem}
\newtheorem{example}[figure]{Example}
\newtheorem{lemma}[figure]{Lemma}
\newtheorem{defn}[figure]{Definition}
\newtheorem{prop}[figure]{Proposition}
\newtheorem{cor}[figure]{Corollary}

\newtheorem{remark}[figure]{Remark}
\begin{document}
\begin{abstract}
We study combinatorics of two generalizations of Dellac configurations.
First, we establish a correspondence between a generalized Dellac configuration with three
parameters and a generalized Dumont permutations.
Secondly, by relaxing conditions on Dellac configurations, we introduce a generalization which 
we call Dellac configurations with boundaries.
We show several recurrence relations for the Poincar\'e polynomials of Dellac configurations 
with boundaries.
\end{abstract}

\maketitle

\section{Introduction}
A Dellac configuration \cite{Del00} is a configuration of dots with $n$ column and $2n$ rows 
satisfying the following conditions: first, each row contains one dot, and 
each column contains two dots, and second, if there is a dot in 
the $i$-th column and the $j$-th row, then $i\le j\le i+n$. 
The number of such configurations is given by the median 
Genocchi numbers \cite{BarDum79,Fei11,Fei12b}. 
These numbers have been studied  extensively in the context of permutations, the Seidel triangles and 
continued fractions \cite{Big18,Dum74,DumZen94,ZenZho06}.
In this paper, we study combinatorial aspects of two generalizations of Dellac 
configurations.
First, given a triplet $(l,m,n)$ with $l\ge 1$, $m\ge2$ and $n\ge1$, 
we define a generalization of Dellac configurations in a rectangle 
with $ln$ columns and $mn$ rows such that 
each row contains $l$ dots, each column contains $m$ dots, and 
a certain condition on the position of dots (see Section \ref{sec:GDC} for details).
We denote by $\mathrm{DC}_{l,m,n}$ the set of such Dellac configurations.
In case of $(l,m,n)=(1,2,n)$, Dumont permuations \cite{DumRan94,Kre97} play a central role 
to study combinatorial properties of Dellac configurations. 
In \cite{HZ99a,HZ99b}, Han and Zeng introduced the statistics $\mathrm{st}(\sigma)$ 
for a Dumont permutation $\sigma$.
In \cite{Fei11,Fei12a,Fei12b}, Feigin studied the degenerate flag varieties and compute the Poincar\'e 
polynomials by use of statistics on the Dellac configurations.
This statistics can be regarded as a generalization of the inversion for permutations, 
and called the inversion of a Dellac configuration.
In \cite{Big14}, Bigeni gave a combinatorial interpretation of $\mathrm{st}(\sigma)$
by showing a bijection between Dellac configurations and normalized Dumont permutations.
As in the case of $(l,m)=(1,2)$, we have a correspondence between a generalized 
Dellac configuration in $\mathrm{DC}_{l,m,n}$ and a generalized Dumont permutation.

We first introduce a generalization of Dumont permutations, and establish 
a correspondence between a configuration in $\mathrm{DC}_{l,m,n}$ and 
a generalized normalized Dumont permutation.
For this purpose, we introduce a condition on multi-permutations which 
we call parity property with two parameters $(l,m)$.
This parity property is a generalization of conditions on Dumont permutations.  
The results on Dellac configurations and normalized Dumont permutations 
are naturally generalized by introducing the statistics $\mathrm{st}(\sigma)$ 
of a multi-permutation $\sigma$. 
We remark that there are other generalizations of Dellac configurations 
such as symmetric Dellac configurations \cite{Big17,BigFei19,BigFei20} corresponding to the symplectic 
degenerate flag varieties \cite{FanFou15,FeiFinLit14}.

We also give a formula to compute the inversions by using a permutation obtained 
from generalized normalized Dumont permutation (see Theorem \ref{thrm:Dyckinv}). 
Feigin introduced a combinatorial model to study normalized Dumont permutation
of the second kind \cite{Fei11}. This model is expressed by tuples $I^{1},\ldots,I^{n-1}$ with 
a certain condition (see Proposition 3.1 in \cite{Fei11}).
We give two generalizations of this combinatorial model: one is a collection of 
tuples with repeated entries (Definition \ref{defn:I}), and the other is tuples without 
repeated entries (Definition \ref{defn:K}). 

In a Dellac configuration, the dot in the coordinate $(i,j)$ satisfies 
the condition $i\le j\le i+n$. 
This condition indicates that there are regions where dots are not 
allowed to be placed. The regions are parameterized by partitions of 
staircase shape.
We relax this condition by changing the staircase partitions to 
general partitions inside the staircase.
To define a Dellac configuration, we need to have two partitions.
We call a configuration with two partitions a Dellac configuration with boundaries.
The configurations with boundaries naturally appear in the computation
of the Poincar\'e polynomials, which we call partition functions 
in Section \ref{sec:DCgb}.
We show that the partition functions satisfy several simple recurrence 
relations in three terms.
These relations are used to express a partition function in terms of 
partition functions characterized by $\delta_{n}\setminus\{i\}$ where 
$\delta_{n}$ is a staircase partition.

The paper is organized as follows.
In Section \ref{sec:GD}, we introduce a generalized 
Dumont permutation $\sigma$ and its statistics $\mathrm{st}(\sigma)$.
In Section \ref{sec:GDC}, generalized Dellac configurations are defined
for a triplet $(l,m,n)$. We also show some fundamental properties of 
them.
In Section \ref{sec:GDCnDp}, we connect the two notions in the 
previous sections: normalized Dumont permutations and generalized 
Dellac configurations. 	
We construct a bijection preserving the inversion number.
In Section \ref{sec:phiDyck}, we study the word $\phi(C)$ obtained from a 
Dellac configuration $C$, and the relation between $\phi(C)$ and 
inversions of $C$ by using Dyck paths.
In Section \ref{sec:embed}, we introduce two maps to embed a generalized 
Dellac configuration with $(l,m,n)$ in the set of Dellac configurations 
with $(l,m,n)=(1,2,n')$.
We propose two combinatorial models bijective to a generalized Dellac 
configuration in Sections \ref{sec:tuple1} and \ref{sec:tuple2}.
In Section \ref{sec:DCgb}, we study Dellac configurations with 
general boundaries.
We show several recurrence relations for them.

\paragraph{Notation}
The function $\lceil x\rceil$ is the ceiling function, 
{\it i.e.}, $\lceil x\rceil:=\min\{y\in\mathbb{Z}:x\le y\}$.
The function $\lfloor x\rfloor$ is the floor function,
{\it i.e.}, $\lfloor x\rfloor:=\max\{y\in\mathbb{Z}: x\ge y\}$.
We denote binomial coefficients by 
$\genfrac{[}{]}{0pt}{}{n}{m}:=\genfrac{}{}{1pt}{}{n!}{m!(n-m)!}$.
The $q$-integer is denoted by $[n]_q:=\sum_{1\le i\le n}q^{i-1}$, 
and $q$-binomial coefficients by 
$\genfrac{[}{]}{0pt}{}{n}{m}_q:=\genfrac{}{}{1pt}{}{[n]_q!}{[m]_q![n-m]_q!}$.

\section{Generalized Dumont permutations}
\label{sec:GD}
We denote by $\mathfrak{S}_{n}$ the set of permutations of the 
set $[n]:=\{1,2,\ldots,n\}$. 
A generalized permutation of order $L:=nl$ is a word $w:=w_{1}w_{2}\ldots w_{L}$
such that an integer $i\in[n]$ appears exactly $l$ times in $w$.
We denote by $\mathfrak{S}_{L}^{l}$ the set of generalized permutations.
Especially, when $l=1$, $\mathfrak{S}_{L}^{1}$ is nothing but the set of permutations 
of order $L$.

The standardization $\mathrm{Std}:\mathfrak{S}_{L}^{l}\rightarrow\mathfrak{S}_{L}$, 
$\sigma^{l}\mapsto\sigma$ is defined as follows.
We replace $l$ ones in $\sigma^{l}$ with $1,2,\ldots,l$ from left to right, 
replace $l$ twos with $l+1,\ldots,2l$ and in general replace $l$ $i$'s with 
$(i-1)l+1,\ldots,il$.
Similarly, the destandardization 
$\mathrm{dStd}^{l}:\mathfrak{S}_{L}\rightarrow\mathfrak{S}_{L}^{l}$, 
$\sigma\mapsto\sigma^{l}$ is defined by 
$\sigma^{l}_{i}=\lceil\sigma_{i}/l\rceil$ for $1\le i\le L$.
For example, let $\sigma:=221313\in\mathfrak{S}_{6}^{2}$. 
Then, we have $\mathrm{Std}(\sigma)=341526\in\mathfrak{S}_{6}$.
Note that the map $\mathrm{dStd}^{l}$ is not injective. 
For example, $\mathrm{dStd}^2(1324)=\mathrm{dStd}^2(2314)=1212$.

Let $w:=w_{1}\ldots w_{n}$ be a word with $n$ letters in 
the alphabet $\mathbb{N}$.
The number of inversions $\mathrm{inv}(w)$ is the number 
of pairs $(i,j)$ such that $i<j$ and $w_{i}>w_{j}$.
When $w$ is a permutation, $\mathrm{inv}(w)$ coincides with 
the standard definition of an inversion.

Given two words $u$ and $v$, we denote by $u\ast v$ the 
concatenation of two words.
More precisely, if $u=u_1\ldots u_{n}$ and $v=v_1\ldots v_{m}$,
we define $u\ast v:=u_1\ldots u_nv_1\ldots v_m$.

We start with the definition of a Dumont permutation \cite{DumRan94}.
\begin{defn}
A Dumont permutation of order $2n$ is a permutation $\sigma\in\mathfrak{S}_{2n}$ 
such that $\sigma(2i)<2i$ and $\sigma(2i-1)>2i-1$ for all $1\le i\le n$.
We denote by $\mathfrak{D}_{2n}$ the set of Dumont permutations.
\end{defn}

For $l\ge2$, we define generalized Dumont permutations as follows.
\begin{defn}
Let $L=2nl$ and $p$ and $q$ be integers such that 
$0\le p\le 2n-1$ and $1\le q\le l$. 
A generalized Dumont permutation of order $L$ is a generalized permutation 
$\sigma\in\mathfrak{S}_{L}^{l}$ such that 
$\sigma(pl+q)>p+1$ for $p$ even and $\sigma(pl+q)<p+1$ for $p$ odd.
We denote by $\mathfrak{D}_{L}^{l}$ the set of these generalized Dumont  permutations.
\end{defn}
By definitions, Dumont permutations are a generalized Dumont permutations for $l=1$.

We define a generalization of a normalized Dumont permutation for a triplet 
$(l,m,n)$.
Given a triplet $(l,m,n)$, we define $L:=l(nm+2(m-1))$ for $mn$ even and 
$L:=l(mn+2(m-1)-1)$ for $mn$ odd.

\begin{defn}
\label{defn:ppm}
Let $\sigma:=\sigma_{1}\sigma_{2}\ldots\sigma_{m}$ be a word of $m$ letters 
and $\rho:=\rho_{1}\rho_{2}\ldots\rho_{m}$ be the parities of $\sigma$ with 
respect to $l$, 
{\it i.e.}, $\rho_{i}\equiv \lceil \sigma_{i}/l \rceil \pmod2$ 
A word $\sigma$ is said to satisfy the {\it parity property} of type $(l,m)$ 
if $\sigma$ satisfies
\begin{itemize}
\item $\rho$ does not contain a decreasing subsequence.
\item if $\rho\in\{0\}^{m}$ or $\rho\in\{1\}^{m}$, 
then $\sigma$ is an increasing sequence.
\item if $\rho_{i}=0$ for $1\le i\le k$ and 
$\rho_{j}=1$ for $k+1\le j\le m$, then
\begin{eqnarray*} 
\sigma_{k+1}<\sigma_{k+2}<\ldots<\sigma_{m}<\sigma_{1}<\ldots<\sigma_{k}.
\end{eqnarray*}
\end{itemize}
\end{defn}
In Definition \ref{defn:ppm}, the first condition implies that 
we have no subsequence $(\rho_i,\rho_{i+1})=(1,0)$ for some $i$ in $\rho$.

\begin{defn}
\label{def:NDP}
A {\it normalized Dumont permutation} of order $L$ of type $(l,m)$ is 
a generalized Dumont permutation $\sigma\in\mathfrak{D}_{L}^{l}$ such 
that 
\begin{itemize}
\item $\sigma(pl+q)=(p+1)/2$ for odd $p$ satisfying $1\le p\le n(m-2)+1$ and $1\le q\le l$.  
\item We have two cases according to the parity of $mn$:
\begin{itemize}
\item For $mn$ even, $\sigma(pl+q)=(L/l+p)/2+1$ for even $p$ 
satisfying $L/l-n(m-2)-2\le p\le L/l-2$ and $1\le q\le l$.
\item For $mn$ odd, $\sigma(pl+q)=(L/l+p)/2+1$ for even $p$ 
satisfying $L/l-n(m-2)+1\le p\le L/l-2$ and $1\le q\le l$.
\end{itemize}
\item 
Let $q:=n(m-2)/2+1$ for $mn$ even and $q:=n(m-2)/2+1/2$ for $mn$ odd.
Let $\pi\in\mathfrak{S}_{L}$ be a permutation of order $L$ and $\pi'_{p}$
be a sequence of integers 
\begin{align*}
\pi'_{p}:=(\pi^{-1}(lq+pm+1),\pi^{-1}(lq+pm+2),\ldots,\pi^{-1}(lq+(p+1)m)),
\end{align*}
for $p\in\{0,1,\ldots, nl-1\}$ satisfying the parity property of type $(l,m)$.
Further, the subsequences
\begin{align*}
\pi((p-1)l+1), \pi((p-1)l+2), \ldots, \pi(pl),  
\end{align*}
for $1\le p\le L/l$ are weakly increasing.

Then, $\sigma$ is written by $\sigma=\mathrm{dStd}^{l}(\pi)$ for some $\pi$.
We say that $\pi$ satisfy the parity property of type $(l,m)$ if and only if 
all $\pi'_{p}$'s satisfy the parity property of type $(l,m)$.
\end{itemize}
We denote by $\mathfrak{D'}_{L}^{l,m}$ the set of normalized Dumont permutations 
of order $L$ and type $(l,m)$.
\end{defn}

\begin{remark}
Two remarks are in order.
\begin{enumerate}
\item
When $(l,m)=(1,2)$, three conditions in Definition \ref{def:NDP} can be 
reduced to the following single condition: 
$\sigma^{-1}(2j)$ and $\sigma^{-1}(2j+1)$ have the same 
parity iff $\sigma^{-1}(2j)<\sigma^{-1}(2j+1)$.
This is nothing but the defining relations of normalized 
Dumont permutations in \cite{Big14}.
\item
The third condition in Definition \ref{def:NDP} implies that 
a generalized normalized Dumont permutation of type $(l,m)$ 
can be regarded as a set of $ln$ generalized normalized Dumont 
permutation of type $(1,m)$ and combined to be a generalized 
Dumont permutation.
\end{enumerate}
\end{remark}

The following proposition is clear from the construction of $\pi$ in 
Definition \ref{def:NDP}.
\begin{prop}
\label{prop:uniNDP}
Let $\nu$ and $\nu'$ be permutations in $\mathfrak{S}_{L}$ satisfying the 
third condition in Definition \ref{def:NDP}.
Suppose $\sigma, \sigma'\in\mathfrak{D}_{L}^{l}$ satisfy the first two 
conditions in Definition \ref{def:NDP},
and $\sigma:=\mathrm{dStd}^l(\nu)$ and $\sigma':=\mathrm{dStd}^l(\nu')$.
Then, the map $\mathrm{dStd}^l$ is injective, {\it i.e.}, $\nu\neq\nu'\Rightarrow\sigma\neq\sigma'$.
\end{prop}

\begin{example}
\label{ex:NDP}
Let $(l,m,n)=(2,3,3)$.
Then, a generalized Dumont permutation 
\begin{align*}
\sigma:=(5,8,1,1,7,10,2,2,4,8,3,4,9,11,5,7,10,11,6,9,12,12,3,6)
\end{align*}
is a normalized Dumont permutation of order $L=24$ of type $(2,3)$.
It is obvious that $\sigma$ satisfies the first two conditions in 
Definition \ref{def:NDP}.
Further, $\sigma=\mathrm{dStd}^{2}(\pi)$, where 
\begin{align*}
\pi:=(10,15,1,2,13,20,3,4,7,16,5,8,18,21,9,14,19,22,11,17,23,24,6,12).
\end{align*}
By straightforward computations, all $\pi'_{p}$'s, $0\le p\le 5$, satisfy 
the parity property of type $(2,3)$.
\end{example}

Given a normalized Dumont permutation $\sigma\in\mathfrak{D'}_{L}^{l,m}$ 
of type $(l,m)$, we define the following statistics.
Let $\pi$ be the permutation satisfying $\sigma=\mathrm{dStd}^l(\pi)$.
The permutation $\pi$ satisfies the third condition in Definition \ref{def:NDP}.
Note that from Proposition \ref{prop:uniNDP}, $\pi$ is unique if $\sigma$ given.
We define a statistic $\mathrm{st}(\sigma)$ following \cite{HZ99a,HZ99b}.  
\begin{defn}
For all $\sigma\in\mathfrak{S}_{L}^{l}$, we define $\mathrm{st}(\sigma)$
through $\pi$ as the number 
\begin{align}
\label{eqn:defst}
\mathrm{st}(\sigma):=
L^2/4-\sum_{p:odd}\sum_{1\le q\le l}\pi(pl+q)-
\mathrm{inv}(\pi^{o})-\mathrm{inv}(\pi^{e})
\end{align}
where $\pi^{o}$ and $\pi^{e}$ are the two partial words of $\pi$. 
They are 
$\pi^{o}:=\{\pi(pl+q): p \text{ is even and } 1\le q\le l\}$ 
and $\pi^{e}:=\{\pi(pl+q): p \text{ is odd and } 1\le q\le l\}$ 
respectively.
\end{defn}

\begin{example}
\label{ex:NDP2}
We consider the same normalized Dumont permutation as in Example \ref{ex:NDP}.
We have 
\begin{align*}
\pi^{e}&=(1,2,3,4,5,8,9,14,11,17,6,12), \\
\pi^{o}&=(10,15,13,20,7,16,18,21,19,22,23,24).
\end{align*}
Then, we have 
\begin{align*}
\sum_{p:odd}\sum_{1\le q\le l}\pi(pl+q)&=\sum_{i}\pi^{e}_{i}=92, \\
\mathrm{inv}(\pi^{o})&=9,\\
\mathrm{inv}(\pi^{e})&=8.
\end{align*}
The statistics $\mathrm{st}(\sigma)$ is given by
\begin{align*}
\mathrm{st}(\sigma)=144-92-9-8=35.
\end{align*}
\end{example}

\section{Generalized Dellac configuration}
\label{sec:GDC}
\subsection{Generalized Dellac configuration}
Let $l\ge1,m\ge2$ and $n\ge1$ be positive integers.
We define a generalization of Dellac configurations with 
a triplet $(l,m,n)$.

\begin{defn}
\label{defn:GDC}
A {\it generalized Dellac configuration} of size $n$ with type 
$(l,m)$ is a tableau $D$ of width $ln$ and height $mn$ which 
contains $lmn$ dots such that 
\begin{itemize}
\item each row contains exactly $l$ dots;
\item each column contains exactly $m$ dots;
\item if there is a dot in the box $(i,j)$ of $D$, 
then $\lceil j/l\rceil\le i\le \lceil j/l\rceil+(m-1)n$.
\end{itemize}
Here, a box $(i,j)$ of $D$ is the one in the $i$-th row 
from bottom to top and $j$-th column from left to right.
\end{defn}

The set of the generalized Dellac configurations of size $n$
with type $(l,m)$ is denoted by $\mathrm{DC}_{l,m,n}$.
When the type $(l,m)$ is obvious from the context, we 
abbreviate $\mathrm{DC}_{n}:=\mathrm{DC}_{l,m,n}$.
A Dellac configuration with type $(1,2)$ is nothing 
but the original definition of a Dellac configuration
studied in \cite{Del00}.

Let $C\in\mathrm{DC}_{n}$ and $d_{k}$, $k=1,2$ be 
dots in $C$ whose Cartesian coordinates in $C$ 
are $(i_k,j_k)$.
An {\it inversion} of $C$ is a pair $(d_1,d_2)$ of
dots such that $i_1<i_2$ and $j_1>j_2$.
We denote by $\mathrm{inv}(C)$ the number of inversions 
of $C$.
Given a dot $d_1$ with Cartesian coordinates $(i_1,j_1)$,
we denote by $l_C(d_1)$ (resp. $r_{C}(d_1)$) the number 
of inversions of $C$ between the dot $d_1$ and another 
dot $d_2$ with $i_2>i_1$ (resp. $i_1>i_2$).

\begin{example}
A Dellac configuration $C$ with $(l,m,n)=(2,2,3)$ is depicted in 
Figure \ref{ex:GDC}.
The number of inversions is $\mathrm{inv}(C)=9$.
\begin{figure}[ht]
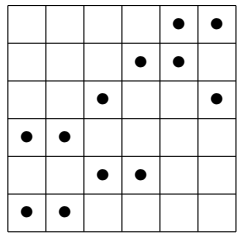

\tikzpic{-0.5}{
\draw(0,0)--(0,6/2)(1/2,0)--(1/2,6/2)(1,0)--(1,6/2)(3/2,0)--(3/2,6/2)
(2,0)--(2,6/2)(5/2,0)--(5/2,6/2)(3,0)--(3,6/2);
\draw(0,0)--(6/2,0)(0,1/2)--(6/2,1/2)(0,1)--(6/2,1)(0,3/2)--(6/2,3/2)
(0,2)--(6/2,2)(0,5/2)--(6/2,5/2)(0,3)--(6/2,3);
\draw(1/4,1/4)node{$\bullet$}(1/4+1/2,1/4)node{$\bullet$}
(1/4+1,3/4)node{$\bullet$}(1/4+1/2+1,3/4)node{$\bullet$}
(1/4,1/4+1)node{$\bullet$}(1/4+1/2,1/4+1)node{$\bullet$}
(1/4+1,3/4+1)node{$\bullet$}(1/4+1/2+2,3/4+1)node{$\bullet$}
(1/4+3/2,3/4+3/2)node{$\bullet$}(1/4+1/2+3/2,3/4+3/2)node{$\bullet$}
(1/4+2,3/4+2)node{$\bullet$}(1/4+1/2+2,3/4+2)node{$\bullet$};
}
\caption{A generalized Dellac configuration of size $3$ with type $(2,2)$.}
\label{ex:GDC}
\end{figure}
\end{example}

\paragraph{\bf Lowest and highest configurations}
Given a triplet $(l,m,n)$, we denote by $C_{0}:=C_{0}(l,m,n)$ the generalized Dellac
configuration of size $n$ defined as follows.
When $j=pl+q$ with $0\le p\le n-1$ and $1\le q\le l$, the boxes $(i,j)$ in $C_{0}$ 
with $pm+1\le i\le (p+1)m$ contain dots.
The configuration $C_{0}$ contains 
$n\genfrac{(}{)}{0pt}{}{m}{2}\genfrac{(}{)}{0pt}{}{l}{2}$ inversions.

Similarly, we denote by $C_{1}:=C_{1}(l,m,n)$ the generalized Dellac configuration
of size $n$ defined as follows.
Let $1\le j\le nl$.
When $1\le i\le n$, the box $(i,j)$ with $i=\lceil j/l\rceil$ has a dot.
When $(m-1)n+1\le i\le mn$, the box $(i,j)$ with $i=(m-1)n+\lceil j/l\rceil$ 
has a dot. 
When $n+k(m-2)+1\le i\le n+(k+1)(m-2)$ with $0\le k\le n-1$, 
the box $(i,j)$ with $\lceil j/l\rceil=n-k$ contains a dot.
The number of inversions for the configuration $C_1$ is  
\begin{align}
\label{eqn:invhighest}
\mathrm{inv}(C_{1})
=\genfrac{(}{)}{0pt}{}{nl}{2}(m-1)
+l^{2}\genfrac{(}{)}{0pt}{}{n}{2}(m-1)(m-2)
+n\genfrac{(}{)}{0pt}{}{l}{2}\genfrac{(}{)}{0pt}{}{m-1}{2}.
\end{align}

We call $C_{0}$ (resp. $C_{1}$) the lowest (resp. highest) configuration.
\begin{example}
The lowest and highest configurations in $\mathrm{DC}_{2,3,2}$ are depicted 
in Figure \ref{fig:lowhigh}.
The number of inversions are $\mathrm{inv}(C_{0})=6$ and $\mathrm{inv}(C_{1})=22$.
\begin{figure}[ht]
\tikzpic{-0.5}{
\draw(0,0)--(0,6/2)(1/2,0)--(1/2,6/2)(1,0)--(1,6/2)(3/2,0)--(3/2,6/2)(2,0)--(2,6/2);
\draw(0,0)--(4/2,0)(0,1/2)--(4/2,1/2)(0,1)--(4/2,1)(0,3/2)--(4/2,3/2)(0,2)--(4/2,2)
(0,5/2)--(4/2,5/2)(0,3)--(4/2,3);
\draw(1/4,1/4)node{$\bullet$}(1/4+1/2,1/4)node{$\bullet$}
(1/4,1/4+1/2)node{$\bullet$}(1/4+1/2,1/4+1/2)node{$\bullet$}
(1/4,1/4+1)node{$\bullet$}(1/4+1/2,1/4+1)node{$\bullet$}
(1/4+1,1/4+3/2)node{$\bullet$}(1/4+3/2,1/4+3/2)node{$\bullet$}
(1/4+1,1/4+2)node{$\bullet$}(1/4+3/2,1/4+2)node{$\bullet$}
(1/4+1,1/4+5/2)node{$\bullet$}(1/4+3/2,1/4+5/2)node{$\bullet$};
} \qquad
\tikzpic{-0.5}{
\draw(0,0)--(0,6/2)(1/2,0)--(1/2,6/2)(1,0)--(1,6/2)(3/2,0)--(3/2,6/2)(2,0)--(2,6/2);
\draw(0,0)--(4/2,0)(0,1/2)--(4/2,1/2)(0,1)--(4/2,1)(0,3/2)--(4/2,3/2)(0,2)--(4/2,2)
(0,5/2)--(4/2,5/2)(0,3)--(4/2,3);
\draw(1/4,1/4)node{$\bullet$}(1/4+1/2,1/4)node{$\bullet$}
(1/4+1,1/4+1/2)node{$\bullet$}(1/4+3/2,1/4+1/2)node{$\bullet$}
(1/4+1,1/4+1)node{$\bullet$}(1/4+3/2,1/4+1)node{$\bullet$}
(1/4,1/4+3/2)node{$\bullet$}(1/4+1/2,1/4+3/2)node{$\bullet$}
(1/4,1/4+2)node{$\bullet$}(1/4+1/2,1/4+2)node{$\bullet$}
(1/4+1,1/4+5/2)node{$\bullet$}(1/4+3/2,1/4+5/2)node{$\bullet$};
}

\caption{$C_{0}$ (left picture) and $C_{1}$ (right picture) in $\mathrm{DC}_{2,3,2}$}
\label{fig:lowhigh}
\end{figure}
\end{example}

\begin{prop}
A Dellac configuration $C$ satisfies 
$\mathrm{inv}(C_{0})\le\mathrm{inv}(C)\le\mathrm{inv}(C_{1})$ and 
equality holds if and only if $C=C_{0}$ or $C=C_{1}$.
\end{prop}

\begin{proof}
We first show that $\mathrm{inv}(C_{0})\le\mathrm{inv}(C)$ and 
the equality holds if and only if $C=C_{0}$.
Let $C_{\triangle}\in\mathrm{C}_{l,m,n}$ be a configuration 
such that there are no dots in the boxes $(i_{1},j_{1})$ and 
$(i_{2},j_{2})$, it has two dots in the boxes $(i_{1},j_{2})$ and 
$(i_{2},j_{1})$ and it has $a_{i}$, $1\le i\le 5$, dots in the 
surrounded regions as in Figure \ref{Conf:sqtri}.
Similarly, let $C_{\square}$ be a configuration such that 
it has no dots in the boxes $(i_{1},j_{2})$ and 
$(i_{2},j_{1})$, it has two dots in the boxes $(i_{1},j_{1})$ and 
$(i_{2},j_{2})$ and a configuration of other dots are the same 
as $C_{\triangle}$.
\begin{figure}[ht]
\tikzpic{-0.5}{
\draw(0,0)--(2,0)(0,1/2)--(2,1/2)(0,3/2)--(2,3/2)(0,2)--(2,2);
\draw(0,0)--(0,2)(1/2,0)--(1/2,2)(3/2,0)--(3/2,2)(2,0)--(2,2);
\draw(1/4,1/4)node{$\square$}(7/4,7/4)node{$\square$};
\draw(1/4,7/4)node{$\triangle$}(7/4,1/4)node{$\triangle$};
\draw(-3/8,1/4)node{$i_1$}(-3/8,7/4)node{$i_{2}$};
\draw(1/4,-3/8)node{$j_1$}(7/4,-3/8)node{$j_2$};
\draw(-3/8,1)node{$\vdots$}(1,-3/8)node{$\cdots$};
\draw(1/4,1)node{$a_{2}$}(7/4,1)node{$a_{4}$}
     (1,1)node{$a_{3}$}(1,7/4)node{$a_{1}$}
     (1,1/4)node{$a_{5}$};
}
\caption{Configurations $C_{\triangle}$ and $C_{\square}$}
\label{Conf:sqtri}
\end{figure}
By a simple calculation we have 
\begin{eqnarray*}
\mathrm{inv}(C_{\triangle})-\mathrm{inv}(C_{\square})
&=&a_{1}+a_{2}+2a_{3}+a_{4}+a_{5}+1, \\
&\ge&1.
\end{eqnarray*}
The operation $\gamma$ which transforms a configuration $C_{\triangle}$ to 
another configuration $C_{\square}$ strictly decreases the number 
of inversions.
It is easy to see that one can obtain the configuration $C_{0}$ 
from $C$ by successive applications of the operation.
Further, we cannot perform such an operation on $C_{0}$.
This implies that $\mathrm{inv}(C_{0})$ is the minimum and it 
is unique by construction of the operation.

Similarly, the inverse operation $\gamma^{-1}$ strictly increases the 
number of inversions.
Any configuration $C$ can be transformed to $C_{1}$
by successive applications of $\gamma^{-1}$.
One can not perform $\gamma^{-1}$ on the configuration $C_{1}$ any more.
By the same reason as in the case of $C_{0}$, $\mathrm{inv}(C_{1})$ 
is the maximum and it is unique. 
\end{proof}

Let $C\in\mathrm{DC}_{l,m,n}$.
We assign a label to a dot in $C$.
We have two cases:

\paragraph{Case A ($mn$ is even)}
For all $1\le i\le mn/2$, a dot in the $i$-th row
of $C$ (from bottom to top) is labeled by 
the integer $e_{i}=2i+n(m-1)+2$, and a dot 
in the $(mn/2+i)$-th row is labeled by the integer 
$e_{n+i}:=2i-1$.

\paragraph{Case B ($mn$ is odd)}
For all $1\le i\le \lfloor mn/2\rfloor$, a dot in the $i$-th row
of $C$ (from bottom to top) is labeled by 
the integer $e_{i}=2i+n(m-2)+1$. 
For all $1\le i\le \lfloor mn/2\rfloor+1$ a dot 
in the $(\lfloor mn/2\rfloor+i)$-th row is labeled by the integer 
$e_{\lfloor mn/2\rfloor+i}:=2i-1$.

We denote by $\mathrm{word}(C)$ 
a word (a sequence of integers) obtained from a label of $C$ by reading 
the label of dots from bottom to top and from left to right.

\begin{remark}
The order of reading the labels of dots to obtain $\mathrm{word}(C)$ 
is different from the one in \cite{Big14}.
To follow the latter order, we have to modify Definition \ref{defn:ppm}  
and the third condition in Definition \ref{def:NDP}.
\end{remark}

\begin{example}
\label{example:GDCLabel}
The label of a Dellac configuration $C\in\mathrm{DC}_{2,3,2}$ is depicted 
in Figure \ref{ex:GDCLabel}.
A word associated with $C$ is $\mathrm{word}(C)=683613\underline{10}158\underline{10}5$.

\begin{figure}[ht]
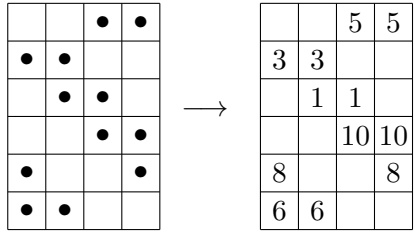


\tikzpic{-0.5}{
\draw(0,0)--(0,6/2)(1/2,0)--(1/2,6/2)(1,0)--(1,6/2)(3/2,0)--(3/2,6/2)(2,0)--(2,6/2);
\draw(0,0)--(4/2,0)(0,1/2)--(4/2,1/2)(0,1)--(4/2,1)(0,3/2)--(4/2,3/2)(0,2)--(4/2,2)
(0,5/2)--(4/2,5/2)(0,3)--(4/2,3);
\draw(1/4,1/4)node{$\bullet$}(1/4+1/2,1/4)node{$\bullet$}
(1/4,1/4+1/2)node{$\bullet$}(1/4+3/2,1/4+1/2)node{$\bullet$}
(1/4+1,1/4+1)node{$\bullet$}(1/4+3/2,1/4+1)node{$\bullet$}
(1/4+1/2,1/4+3/2)node{$\bullet$}(1/4+1,1/4+3/2)node{$\bullet$}
(1/4,1/4+2)node{$\bullet$}(1/4+1/2,1/4+2)node{$\bullet$}
(1/4+1,1/4+5/2)node{$\bullet$}(1/4+3/2,1/4+5/2)node{$\bullet$};
}
$\longrightarrow$
\tikzpic{-0.5}{
\draw(0,0)--(0,6/2)(1/2,0)--(1/2,6/2)(1,0)--(1,6/2)(3/2,0)--(3/2,6/2)(2,0)--(2,6/2);
\draw(0,0)--(4/2,0)(0,1/2)--(4/2,1/2)(0,1)--(4/2,1)(0,3/2)--(4/2,3/2)(0,2)--(4/2,2)
(0,5/2)--(4/2,5/2)(0,3)--(4/2,3);
\draw(1/4,1/4)node{$6$}(1/4+1/2,1/4)node{$6$}
(1/4,1/4+1/2)node{$8$}(1/4+3/2,1/4+1/2)node{$8$}
(1/4+1,1/4+1)node{$10$}(1/4+3/2,1/4+1)node{$10$}
(1/4+1/2,1/4+3/2)node{$1$}(1/4+1,1/4+3/2)node{$1$}
(1/4,1/4+2)node{$3$}(1/4+1/2,1/4+2)node{$3$}
(1/4+1,1/4+5/2)node{$5$}(1/4+3/2,1/4+5/2)node{$5$};
}

\caption{Label of a Dellac configuration}
\label{ex:GDCLabel}
\end{figure}

\end{example}

\subsection{Definition of \texorpdfstring{$\phi$}{phi}}
\label{def:phi}
Given a triplet $(l,m,n)$, we define $L:=l(nm+2(m-1))$ for 
$mn$ even and $L:=l(mn+2(m-1)-1)$ for $mn$ odd.
Recall that $\mathfrak{S}_{L}^{l}$ is the set of generalized 
permutations $\sigma$ of length $L$ such that 
an integer $1\le i\le L/l$ appears exactly $l$ times 
in $\sigma$.

We define $\phi:\mathrm{DC}_{l,m,n}\rightarrow\mathfrak{S}_{L}$ 
by composing maps $\phi_{1}:\mathrm{DC}_{l,m,n}\rightarrow\mathfrak{S}_{L}^{l}$, 
$\phi_{2}:\mathfrak{S}_{L}^{l}\rightarrow\mathfrak{S}_{L}$, 
{\it i.e.}, $\phi:=\phi_{2}\circ\phi_1$.
Roughly speaking, a map $\phi_2$ is an inverse of standardization of a generalized 
permutation in $\mathfrak{S}_{L}^{l}$.
Note that inverse and standardization do not commute in $\mathfrak{S}_{L}^{l}$.

The map $\phi_{1}:\mathrm{DC}_{l,m,n}\rightarrow\mathfrak{S}_{L}^{l}$
is defined by 
\begin{eqnarray*}
\phi_{1}(C):=w_{1}\ast \mathrm{word}(C)\ast w_{2},
\end{eqnarray*}
where 
\begin{eqnarray*}
w_1&:=&2^{l}4^{l}\cdots (n(m-2)+2)^{l}, \\
w_{2}&:=&(mn+1)^{l}(mn+3)^{l}\cdots(2n(m-1)+1)^{l},
\end{eqnarray*}
for $mn$ even, and 
\begin{eqnarray*}
w_1&:=&2^{l}4^{l}\cdots ((m-2)n+1)^{l}, \\
w_2&:=&(mn+2)^{l}(mn+4)^{l}\cdots(2(m-1)n-1)^{l},
\end{eqnarray*}
for $mn$ odd. 

The map $\phi_{2}:\mathfrak{S}_{L}^{l}\rightarrow\mathfrak{S}_{L}$, 
$\alpha:=\alpha_{1}\ldots \alpha_{L}\mapsto \beta=\beta_{1}\ldots\beta_{L}$
is defined as follows.
Recall that an integer $1\le k\le L/l$ appears exactly $l$ times in $\alpha$.
Suppose that $\alpha_{i_1}=\alpha_{i_2}=\ldots,=\alpha_{i_l}=k$ and 
$i_1<i_2<\ldots<i_{l}$.
Then, we define $\beta_{l(k-1)+p}:=i_{p}$ for $1\le p\le l$.

\begin{example}
We consider the same Dellac configuration as in Example \ref{example:GDCLabel}.
By definition, we have 
\begin{eqnarray*}
\phi_{1}(C)&=&2244683613\underline{10}158\underline{10}57799, \\
\phi(C)&=&(9,12,1,2,7,10,3,4,13,16,5,8,17,18,6,14,19,20,11,15).
\end{eqnarray*}
\end{example}

\subsection{Alternative algorithm}
\label{sec:GDCalt}
Let $C\in\mathrm{DC}_{l,m,n}$ be a generalized Dellac configuration.
We enumerate dots in $C$ by $1,2,\ldots,lmn$ 
from bottom to top and from left to right.
We denote by $d_{i}, 1\le i\le lmn$, the dot enumerated by an integer $i$.
A word $\tau_C\in\mathfrak{S}_{lmn}$ is defined by reading integers associated 
with dots from left to right and from bottom to top.

The following lemma is clear from the construction of the word 
$\tau_{C}$.
\begin{lemma}[Lemma 2.5 in \cite{Big14}]
\label{lemma:invtau}
Let $C\in\mathrm{DC}_{l,m,n}$ and $1\le p<q\le lmn$.
Then, a pair $(d_p,d_q)$ of dots is an inversion of $C$ 
if and only if $(p,q)$ is an inversion of $\tau_C$, namely, 
$q$ is left to $p$ in $\tau_C$.
\end{lemma}
The following proposition is a direct consequence of Lemma \ref{lemma:invtau}.
\begin{prop}
Let $C\in\mathrm{DC}_{l,m,n}$. Then, we have 
$\mathrm{inv}(C)=\mathrm{inv}(\tau_C)$.
\end{prop}

\begin{example}
Let $C$ be a generalized Dellac configuration as in Example \ref{example:GDCLabel}.
\begin{eqnarray*}
\tikzpic{-0.5}{
\draw(0,0)--(0,6/2)(1/2,0)--(1/2,6/2)(1,0)--(1,6/2)(3/2,0)--(3/2,6/2)(2,0)--(2,6/2);
\draw(0,0)--(4/2,0)(0,1/2)--(4/2,1/2)(0,1)--(4/2,1)(0,3/2)--(4/2,3/2)(0,2)--(4/2,2)
(0,5/2)--(4/2,5/2)(0,3)--(4/2,3);
\draw(1/4,1/4)node{$1$}(1/4+1/2,1/4)node{$4$}
(1/4,1/4+1/2)node{$2$}(1/4+3/2,1/4+1/2)node{$10$}
(1/4+1,1/4+1)node{$7$}(1/4+3/2,1/4+1)node{$11$}
(1/4+1/2,1/4+3/2)node{$5$}(1/4+1,1/4+3/2)node{$8$}
(1/4,1/4+2)node{$3$}(1/4+1/2,1/4+2)node{$6$}
(1/4+1,1/4+5/2)node{$9$}(1/4+3/2,1/4+5/2)node{$12$};
}
\longrightarrow \tau_C=142\underline{10}7\underline{11}58369\underline{12}.
\end{eqnarray*}
The number of inversions is $\mathrm{inv}(C)=\mathrm{inv}(\tau_C)=19$.
\end{example}

We enumerate the dots in $C$ by $1,2,\ldots,lmn$ from left to right 
and from bottom to top. 
We denote by $f_{i}, 1\le i\le lmn$, the dot enumerated by an integer $i$.
We have 
\begin{prop}[Proposition 2.6 in \cite{Big14}]
For all $1\le i\le lmn$, we have $\tau_{C}(i)=i+l_C(f_i)-r_C(f_i)$,
where 
\begin{align*}
l_C(f_i)&:=\#\{k>i | \tau_{C}(k)<\tau_{C}(i)\}, \\
r_C(f_i)&:=\#\{k<i | \tau_{C}(k)>\tau_{C}(i)\}.
\end{align*}
\end{prop}

\section{Dellac configurations and normalized Dumont permutations}
\label{sec:GDCnDp}
\subsection{Switch of a generalized Dellac configuration}
In this subsection, we introduce an operation on a generalized 
Dellac configuration, called switch following \cite{Big14}.

Let $C\in\mathrm{DC}_{l,m,n}$ and $i\in\{1,2,\ldots,lmn-1\}$.
We denote by $\mathrm{Sw}^{i}(C)$ the tableau obtained by 
switching two dots $d_{i}$ and $d_{i+1}$. 
Here, switch means that when the Cartesian coordinates of 
$d_{k}$, $k=i,i+1$, are $(p_{k},q_{k})$, then we delete 
two dots $d_{i}$ and $d_{i+1}$ from $C$ and add two 
dots whose coordinates are $(p_{i},q_{i+1})$ and 
$(p_{i+1},q_{i})$.
If $\mathrm{Sw}^{i}(C)$ is in $\mathrm{DC}_{l,m,n}$, we say 
that $C$ is switchable at $i$.

\begin{remark}
A generalized Dellac configuration $C\in\mathrm{DC}_{l,m,n}$ 
contains $l$ dots in a row. Therefore, in some cases, 
we can delete two dots, however, cannot add two dots in the switch operation.
In this case, we define $\mathrm{Sw}^{i}(C):=\emptyset$.
Further, there may be several ways of switches in $C$ in the 
same row.  
\end{remark}

Suppose that there exist no dots in the coordinates
$(i,j)$ with $i_1\le i\le i_2$ and $j_1\le j\le j_2$
except two coordinates $(i_1,j_1)$ and $(i_2,j_2)$ 
in matrix notation.
An elementary switch $\eta_{(i_1,i_2)}$ is a switch of two dots whose 
coordinates are $(i_1,j_1)$ and $(i_2,j_2)$. 
The action of $\eta_{(i_1,i_2)}$ results in two dots whose 
coordinates are $(i_1,j_2)$ and $(i_2,j_1)$.
The following proposition is obvious from the definition of 
an elementary switch.

\begin{prop}
\label{prop:eswitch}
Let $C\in\mathrm{DC}_{l,m,n}$ and $\eta_{i}$ be a elementary 
switch.
Assume $\eta_{i}(C)\neq 0$.
\begin{enumerate}
\item $\mathrm{inv}(\eta_{i}(C))=\mathrm{inv}(C)-1$. 
\item Let $C_0$ is the lowest configuration in $\mathrm{DC}_{l,m,n}$.
Given $C$, there exists a set of pairs of integers 
$I_{i}:=(i_{2i-1},i_{2i})$  
such that 
\begin{align*}
C=\eta_{I_k}\circ\eta_{I_{k-1}}\circ\cdots\circ\eta_{I_{1}}(C_0).
\end{align*}
\end{enumerate}
\end{prop}

\subsection{Bijection \texorpdfstring{$\varphi$}{varphi}}
In this subsection, we will construct a bijection between 
a normalized Dumont permutation $\sigma$ of type $(l,m)$ in $\mathfrak{D'}_{L}^{l,m}$ 
and a generalized Dellac configuration $C$ in $\mathrm{DC}_{l,m,n}$.
Further, this bijection has a property that connects the statistics $\mathrm{st}$ 
of $\sigma$ with the inversion number of $C$.
Below, we fix a triplet $(l,m,n)$.

The main purpose of this section is to prove the following theorem. 
This is a generalization of the correspondence between Dellac configurations 
with $(l,m)=(1,2)$ and Dumont permutations studied in \cite{Big14}.
\begin{theorem}
\label{thrm:varphi}
There exists a bijection $\varphi:\mathrm{DC}_{l,m,n}\rightarrow\mathfrak{D'}_{L}^{l,m}$ 
such that 
\begin{align}
\label{eqn:st11}
\mathrm{st}(\psi(C))=\genfrac{(}{)}{0pt}{}{L/2}{2}-\mathrm{inv}(C)
\end{align}
where $\phi:=\phi_2\circ\phi_{1}(C)$ and $\varphi=\mathrm{dStd}^{l}\circ\phi$.
\end{theorem}

\begin{remark}
The map $\phi$ in Theorem \ref{thrm:varphi} is defined in Section \ref{def:phi}, 
and the map $\mathrm{dStd}^{l}$ is defined in Section \ref{sec:GD}.
We define the map $\phi_{3}:\mathfrak{S}_{L}^{l}\rightarrow\mathfrak{S}_{L}^{l}$, 
$\alpha:=\alpha_{1}\ldots \alpha_{L}\mapsto \beta=\beta_{1}\ldots\beta_{L}$
as follows.
Recall that an integer $1\le k\le L/l$ appears exactly $l$ times in $\alpha$.
Suppose that $\alpha_{i_1}=\alpha_{i_2}=\ldots,=\alpha_{i_l}=k$ and 
$i_1<i_2<\ldots<i_{l}$.
Then, we define $\beta_{l(k-1)+p}:=\lceil i_{p}/l\rceil$ for $1\le p\le l$.
We have $\varphi=\phi_{3}\circ\phi_{1}$.
Equivalently, we have $\phi_3=\mathrm{dStd}^{l}\circ\phi_2$.
\end{remark}

Before proceeding to the proof of Theorem \ref{thrm:varphi},
we construct the bijection $\varphi$ in the following three 
propositions.

\begin{prop}
\label{prop:DCNDP}
For all $C\in\mathrm{DC}_{l,m,n}$, the generalized permutation 
$\phi_{3}\circ\phi_{1}(C)$ is a normalized Dumont permutation 
of type $(l,m)$.
\end{prop}
\begin{proof}
Let $\pi:=\phi_{3}\circ\phi_{1}(C)$ for $C\in\mathrm{DC}_{l,m,n}$.
From Definition \ref{def:NDP}, we have three conditions which 
characterize a normalized Dumont permutation of type $(l,m)$.
It is obvious that the first and second conditions are satisfied 
by $\pi$. 
In fact, given $C\in\mathrm{DC}_{l,m,n}$, we concatenate three words
$w_1$, $\mathrm{word}(C)$ and $w_2$ into a generalized 
permutation in $\mathfrak{S}_{L}^{l}$.
The two words $w_1$ and $w_2$ satisfy the two conditions when we 
construct $\pi$.

We show that $\pi$ satisfy the third condition, or equivalently,
the parity property of type $(l,m)$.
Let $c$ be the left-most column in the generalized Dellac configuration 
$C$.
By definition of generalized Dellac configurations, we can not put 
a dot in the $i$-th row from top for $1\le i\le n-1$ in $c$.
When $mn$ even (resp. odd), the label of the highest dot is 
less than or equal to $(m-2)n+1$ (resp. $(m-2)n+2$).
On the other hand, the label of the lowest dot in $c$ is 
$(m-2)n+4$ (resp. $(m-2)n+3$) for $mn$ even (resp. odd).
Let $l(i)$, $1\le i\le m$, be the label of the $i$-th dot in $c$ 
from bottom to top.
Since we have $l(m)<l(1)$, there exists a unique $k$, $1\le k\le m$,
such that 
\begin{align}
\label{eqn:pplm}
l(k+1)<l(k+2)<\cdots<l(m)<l(1)<l(2)<\cdots<l(k).
\end{align}
Since $\pi$ is obtained by maps $\phi_{3}\circ\phi_{1}$, 
the condition (\ref{eqn:pplm}) implies that 
$\pi$ satisfies the parity property of type $(l,m)$.
The condition that 
\begin{align*}
\pi((p-1)l+1), \pi((p-1)l+2), \ldots, \pi(pl),
\end{align*} 
for $1\le p\le L/l$ are weakly increasing is 
obviously satisfied by the definition of $\phi_{3}$.

Since $\pi$ satisfies the all conditions, 
$\pi$ is a normalized Dumont permutation of 
type $(l,m)$.
\end{proof}

\begin{prop}
\label{prop:DCst}
Let $C\in\mathrm{DC}_{l,m,n}$. 
Then, we have 
\begin{align}
\label{eqn:stinv}
\mathrm{st}(\phi(C))=\genfrac{(}{)}{0pt}{}{L/2}{2}-\mathrm{inv}(C).
\end{align}
\end{prop}
\begin{proof}
We first show that Eqn. (\ref{eqn:stinv}) holds for $C=C_1$, where
$C_{1}$ is the highest configuration defined in Section \ref{sec:GDC}.
Since $L$ is explicitly given by $(l,m,n)$, the right hand side of Eqn. (\ref{eqn:stinv}) 
can be calculated easily by use of Eqn. (\ref{eqn:invhighest}).

We compute the left hand side of Eqn. (\ref{eqn:stinv}) for $C_{1}$.
We consider the $m$ odd and $n$ even case, since calculations for
other cases are similar.
We first compute $\mathrm{inv}(\pi^{e})$ and $\mathrm{inv}(\pi^{o})$.
The number $\mathrm{inv}(\pi^{e})$ is the number of inversions 
among the dots from bottom to the $mn/2$-th row.
Since $C_1$ is symmetric under the rotation, we have 
$\mathrm{inv}(\pi^e)=\mathrm{inv}(\pi^{o})$.
In $C_1$, the dots in from the bottom row to the $n$-th row have no
inversions. 
The value $\mathrm{inv}(\pi^e)$ is the sum of the following two values.
The first one is the sum of inversions among the $il+j$-th columns 
where $i\in[n/2,n-1]$ is fixed and $1\le j\le l$.
The second one is the sum of inversions of dots where 
one dot is in the $i_1l+j_1$-th column and the other dot is in the 
$i_2l+j_2$-th column where $i_1$ and $i_2$ are distinct in $[n/2,n-1]$
and $j_1,j_2\in[1,l]$.
Then, by a straightforward computation, we have 
\begin{align}
\mathrm{inv}(\pi^e)
=\genfrac{}{}{1pt}{}{n}{2}\genfrac{(}{)}{0pt}{}{m-1}{2}\genfrac{(}{)}{0pt}{}{l}{2}
+l^{2}(m-1)(m-2)\genfrac{(}{)}{0pt}{}{n/2}{2}.
\end{align}

We compute the second contribution in Eqn. (\ref{eqn:defst}).
When we construct a normalized Dumont permutation from a 
generalized Dellac configuration $C_1$, we attach a word $w_1$
consisting of even integers.  
Since we consider $\pi(pl+q)$ such that $p$ is odd and $1\le q\le l$,
$w_1$ induces a simple sequence $(1,2,\ldots ,\alpha-1)$ 
in $\pi$ where $\alpha=(n(m-2)/2+1)l+1$.

In $C_1$, we have $l$ dots in the $i$-th row with $1\le i\le n$.
By definition of the word $\mathrm{word}(C_1)$, the left bottom dot 
induces $\alpha$ in $\pi$.
In $i$-th row with $1\le i\le mn/2$, $l$ dots are next to each other and 
they induce a sequence in $\pi$: 
\begin{align*}
(\alpha+(k-1)ml,\alpha+((k-1)l+1)m,\ldots,\alpha+(kl-1)m),
\end{align*}
in the $k$-th row with $1\le k\le n$.

Let $p\in[1,m-2]$ be an integer such that $p\equiv k-n\pmod{m-2}$.
The dots in the $k$-th row and $q$-th column with $n+1\le k\le mn/2$ and 
$nl/2\le q\le nl$ induce a sequence in $\pi$ of the form
\begin{align*}
(\alpha+(k'-1)ml+p,\alpha+((k'-1)l+1)m+p,\ldots,\alpha+(k'l-1)m+p),
\end{align*}
where $k':=\lceil q/l\rceil$.

By taking a sum of these induced integers in $\pi$,
we have 
\begin{align*}
\sum_{p:odd}\sum_{1\le q\le l}\pi(pl+q)
&=\sum_{i=1}^{\alpha-1}i
+\sum_{k=1}^{n/2}\left(\alpha l-m\genfrac{(}{)}{0pt}{}{l+1}{2}+ml^{2}k\right) \\
&+\sum_{k=n/2+1}^{n}\left(\alpha l-m\genfrac{(}{)}{0pt}{}{l+1}{2}+ml^{2}k\right)(m-1)
+\genfrac{}{}{1pt}{}{nl}{2}\sum_{i=1}^{m-2}i, \\
&=
\genfrac{(}{)}{0pt}{}{\alpha}{2}
+\genfrac{}{}{1pt}{}{nm}{2}\left(\alpha l-m\genfrac{(}{)}{0pt}{}{l+1}{2}\right)
+ml^{2}\genfrac{(}{)}{0pt}{}{n/2+1}{2} \\
&+\genfrac{}{}{1pt}{}{nl^2}{2}\left(\genfrac{}{}{1pt}{}{3n}{2}+1\right)\genfrac{(}{)}{0pt}{}{m}{2}
+\genfrac{}{}{1pt}{}{nl}{2}\genfrac{(}{)}{0pt}{}{m-2}{2}.
\end{align*}

Substituting these expressions into Eqn. (\ref{eqn:defst}), we obtain that the left hand side of 
Eqn. (\ref{eqn:stinv}) is equal to the right hand side by a straightforward computation.

To prove Eqn. (\ref{eqn:stinv}) for a general configuration $C$,
it is enough to prove that 
\begin{align}
\label{eqn:st1}
\mathrm{st}(\phi(\eta_{i}(C)))=\mathrm{st}(\phi(C))+1,
\end{align}
for an elementary switch $\eta_{i}$, 
since from Proposition \ref{prop:eswitch} we have 
$\mathrm{inv}(\eta_{i}(C))=\mathrm{inv}(C)-1$.

Let $d_1$  and $d_2$ be two dots whose coordinate is $(i,j_1)$
and $(i+1,j_2)$ with $j_1<j_2$ respectively.
We assume that there exists $p$ dots below $d_2$ and $q$ dots 
above $d_1$.
Further, there exist no dots in the coordinates $(i,j+1)$ and 
$(i+1,j)$.
We have three cases: a) $1\le i\le mn/2-1$, b) $i=mn/2$, and 
c) $nm/2+1\le i\le mn-1$.
Let $\pi:=\phi(C)$ and $\pi':=\phi(\eta_{i}(C))$, 
and $A(\pi):=\sum_{r: odd}\sum_{1\le s\le l}\pi(rl+s)$.
Since the proof of c) is similar to that of a) by symmetry, 
we consider only cases a) and b).

Case a). Since $i\le mn/2-1$, we have 
\begin{align*}
\mathrm{inv}(\pi'^o)&=\mathrm{inv}(\pi^{o}), \\
\mathrm{inv}(\pi'^e)&=\mathrm{inv}(\pi^{e})+1, \\
A(\pi')&:=A(\pi).
\end{align*}
Thus, we have Eqn. (\ref{eqn:st1}).

Case b). Since $i=mn/2$, the numbers $\pi'^o$ and $\pi'^e$ of inversions
are different from those for $\pi^o$ and $\pi^{e}$. 
Similarly, $A(\pi')$ is also different from $A(\pi)$.
We have 
\begin{align*}
\mathrm{inv}(\pi'^o)&=\mathrm{inv}(\pi^{o})+q, \\
\mathrm{inv}(\pi'^e)&=\mathrm{inv}(\pi^{e})+p, \\
A(\pi')&=A(\pi)-(p+q+1),
\end{align*}
which implies Eqn. (\ref{eqn:st1}).
This completes the proof.
\end{proof}

We construct a map $\psi$, which is the inverse of the map $\varphi$.
\begin{defn}
The map $\psi:\mathfrak{D'}_{L}^{l,m}\rightarrow\mathrm{DC}_{l,m,n}$, 
$\sigma\mapsto T$, is defined as follows.
Let $\pi$ be a permutation such that $\sigma=\mathrm{dStd}^{l}(\pi)$ 
as in Definition \ref{def:NDP}.
Let $p\in\{0,1,\ldots,n-1\}$ and $q\in\{1,2,\ldots,l\}$. 
The $(pl+q)$-th column contains $m$ dots labeled by 
$\lceil\pi^{-1}(r+(pl+q-1)m+k)/l\rceil$, $1\le k\le m$, 
where $r=l(n(m-2)/2+1)$ for $mn$ even and $r=l(n(m-2)/2+1/2)$ for 
$mn$ odd.
\end{defn}

\begin{prop}
\label{prop:tabGDC}
The tableau $\psi(\sigma)$, $\sigma\in\mathfrak{D'}_{L}^{l,m}$, is 
a generalized Dellac configuration.
\end{prop}
\begin{proof}
Recall the definition of normalized Dumont permutations 
in Definition \ref{def:NDP}.
Since $\sigma\in\mathfrak{D'}_{L}^{l,m}$ and $\pi$ is a permutation
such that $\sigma=\mathrm{dStd}^{l}(\pi)$, $\pi$ satisfies the parity 
property of type $(l,m)$.
The given map $\psi$ is compatible with Definition \ref{def:NDP},
which implies $\psi(\sigma)$ is a generalized Dellac configuration
for $(l,m,n)$.
\end{proof}

\begin{proof}[Proof of Theorem \ref{thrm:varphi}]
From Proposition \ref{prop:DCNDP} and Proposition \ref{prop:DCst},
the map $\psi$ is a map from a generalized Dellac configuration 
to a normalized permutation satisfying Eqn. (\ref{eqn:st11}).
Proposition \ref{prop:tabGDC} gives the inverse of $\psi$, which 
implies that $\psi$ is bijective.
\end{proof}

\begin{example}
We consider the same normalized Dumont permutation as Example \ref{ex:NDP}.
Let $\sigma$ and $\pi$ be generalized permutations in Example \ref{ex:NDP}.
From Proposition \ref{prop:tabGDC}, the corresponding generalized Dellac 
configuration is in Figure \ref{fig:tabGCD}.
\begin{figure}[ht]
$C=$\tikzpic{-0.5}{[scale=0.5]
\draw(0,0)--(0,9)(1,0)--(1,9)(2,0)--(2,9)(3,0)--(3,9)(4,0)--(4,9)(5,0)--(5,9)(6,0)--(6,9);
\draw(0,0)--(6,0)(0,1)--(6,1)(0,2)--(6,2)(0,3)--(6,3)(0,4)--(6,4)(0,5)--(6,5)(0,6)--(6,6)(0,7)--(6,7)
(0,8)--(6,8)(0,9)--(6,9);
\draw(1/2,1/2)node{$\bullet$}(1/2,7/2)node{$\bullet$}(1/2,13/2)node{$\bullet$}
(3/2,1/2)node{$\bullet$}(3/2,3/2)node{$\bullet$}(3/2,9/2)node{$\bullet$}
(5/2,5/2)node{$\bullet$}(5/2,7/2)node{$\bullet$}(5/2,11/2)node{$\bullet$}
(7/2,3/2)node{$\bullet$}(7/2,9/2)node{$\bullet$}(7/2,13/2)node{$\bullet$}
(9/2,5/2)node{$\bullet$}(9/2,15/2)node{$\bullet$}(9/2,17/2)node{$\bullet$}
(11/2,11/2)node{$\bullet$}(11/2,15/2)node{$\bullet$}(11/2,17/2)node{$\bullet$};
}
\caption{A generalized Dellac configuration in $\mathrm{DC}_{2,3,3}$}
\label{fig:tabGCD}
\end{figure}
Since the inversion number is $31$, we have 
\begin{align*}
\genfrac{(}{)}{0pt}{}{L/2}{2}-\mathrm{inv}(C)=\genfrac{(}{)}{0pt}{}{12}{2}-31=35,
\end{align*}
which is equal to $\mathrm{st}(\sigma)$ (see Example \ref{ex:NDP2}).
\end{example}

\section{Characterization of \texorpdfstring{$\phi$}{phi} and Dyck paths}
\label{sec:phiDyck}
\subsection{Dyck paths and inversions}
A {\it Dyck} path of length $2N$ is a lattice path from 
the origin $p_{0}:=(0,0)$ to $p_{2N}:=(2N,0)$ with 
up steps $(1,1)$ and down steps $(1,-1)$, which 
does not go below the horizontal line $y=0$.
We denote by $p_{i-1}, p_{i}\in\mathbb{N}^{2}$ 
the lattice points connected by the $i$-th step in a 
Dyck path. 
A Dyck path of length $2N$ is said to be the highest path 
if it consists of $N$ up steps and successively $N$ down steps. 

Let $\xi:=\xi_{1}\ldots\xi_{N}$ be an increasing sequence 
of positive integers of length $N$ such that $\xi_{i}\le 2i-1$.
We have a bijection between a Dyck path and $\xi$ as follows.
When an integer $i\in\{1,2,\ldots,2N\}$ appears in $\xi$, 
the $i$-th step of a Dyck path is set to be an up step.
Reversely, when the $i$-th step of a Dyck path is an up step,
$i$ appears in $\xi$. The condition $\xi_{i}\le 2i-1$ comes from 
the fact that a Dyck path is above the horizontal line.
We denote by $\mathcal{D}(\xi)$ the Dyck path corresponding 
to $\xi$.
\begin{defn}
The statistics $\mathrm{Area}(\xi)$ is defined as the number 
of unit boxes above $\mathcal{D}(\xi)$ and below the highest
Dyck path.
\end{defn}

Recall that $\phi:=\phi(C)$ a permutation of order $L$.
We devide $\phi$ into two subwords $\phi^{o}$ and $\phi^{e}$.
We define 
$\phi^{o}:=\phi^{o}_1\ldots\phi^{o}_{L/2}$ 
(resp. $\phi^{e}:=\phi^{e}_1\ldots\phi^{e}_{L/2}$) 
such that 
$\phi^{0}_{i}=\phi_{2lp+q}$ (resp. $\phi^{e}_{i}=\phi_{l(2p+1)+q}$) where $i$ is uniquely written as 
$i=pl+q$ with $0\le p$ and $1\le q\le l$.
We define $\phi^{e}_{<}$ as a unique increasing sequence obtained from 
$\phi^{e}$ by sorting in the lexicographic order.

\begin{theorem}
\label{thrm:Dyckinv}
Let $C\in\mathrm{DC}_{l,m,n}$ and $\phi:=\phi(C)$. 
Then, we have 
\begin{eqnarray}
\label{invDyck}
\mathrm{inv}(C)=\mathrm{Area}(\phi^{e}_{<})+\mathrm{inv}(\phi^{e})+
\mathrm{inv}(\phi^{o}).
\end{eqnarray}
\end{theorem}

\begin{example}
Let $\phi^{e}=(1,2,3,5,4,9,6,13)$ and 
$\phi^{o}=(7,11,8,10,12,14,15,16)$.
Then, since $\phi^{e}_{<}=(1,2,3,4,5,6,9,13)$, 
$\mathrm{Area}(\phi^{e}_{<})=7$. 
The generalized permutation is given by
\begin{eqnarray*}
\phi^{-1}=(2,2,4,6,4, 8,1, 3, 6, 3,1,5,8,5,7,7).
\end{eqnarray*}
The generalized Dellac configuration $C$ corresponding 
to $\phi^{-1}$ is 
\begin{eqnarray*}
C=\tikzpic{-0.5}{
\draw(0,0)--(0,3)(1/2,0)--(1/2,3)(1,0)--(1,3)(3/2,0)--(3/2,3)
(2,0)--(2,3)(5/2,0)--(5/2,3)(3,0)--(3,3);
\draw(0,0)--(3,0)(0,1/2)--(3,1/2)(0,1)--(3,1)(0,3/2)--(3,3/2)
(0,2)--(3,2)(0,5/2)--(3,5/2)(0,3)--(3,3);
\draw(1/4,1/4)node{$\bullet$}(3/4,1/4)node{$\bullet$}
(1/4,3/4)node{$\bullet$}(7/4,3/4)node{$\bullet$}
(3/4,5/4)node{$\bullet$}(11/4,5/4)node{$\bullet$}
(5/4,7/4)node{$\bullet$}(9/4,7/4)node{$\bullet$}
(5/4,9/4)node{$\bullet$}(7/4,9/4)node{$\bullet$}
(9/4,11/4)node{$\bullet$}(11/4,11/4)node{$\bullet$};
}.
\end{eqnarray*}
We have $\mathrm{inv}(\phi^{e})=\mathrm{inv}(\phi^{o})=2$ and 
$\mathrm{inv}(C)=11$.
\end{example}

\begin{proof}[Proof of Theorem \ref{thrm:Dyckinv}]
The map $\phi^{-1}(C)$ consists of labels of a generalized 
Dellac configuration $C\in\mathrm{DC}_{l,m,n}$ and two words 
$w_1$ and $w_2$.
When an integer $i\ge l(m-1)+1$ appears in $\phi^{e}$, 
the label of the dot $d_{i}$ is even.
Similarly, if an integer $1\le i\le L$ does not appear in $\phi^{e}$,
which is equivalent to that $i$ appears in $\phi^{o}$, 
then the label of the dot $d_{i}$ is odd. 
Let $i$ and $j$ be two integers satisfying the following two conditions: 
1) $i<j$ and 2) $i\in\phi^{o}$ and $j\in\phi^{e}$. 
Then, by the definition of an inversion of $C$, the pair of 
two dots $(d_{i},d_{j})$ is an inversion.
Therefore, the function $\mathrm{Area}(\phi^{e}_{<})$ counts the 
number of inversions $(d_{i},d_{j})$ such that $i\in\phi^{o}$ and $j\in\phi^{e}$.

Let $i$ and $j$ be the two integers such that $l(m-1)+1\le i<j$ and 
$i$ and $j$ appears in $\phi^{e}$.
Let $r_{i}$ (resp. $r_{j}$) be an integer such that $\phi^{e}_{r_{i}}=i$ 
(resp. $\phi^{e}_{r_{j}}=j$). 
The number of inversions of $\phi^{e}$ is the number of pairs $(r_{i},r_{j})$ 
satisfying $r_{i}>r_{j}$.
Let $(f_{i},f_{j})$ be a pair of dots forming an inversion of $C$, $i>j$ and 
their labels are even.
Then, $i$ and $j$ appear in $\phi^{e}$ and satisfy $r_{i}>r_{j}$.
The number of inversions $\mathrm{inv}(\phi^{e})$ in $\phi^{e}$ is 
equal to the number of inversions of $C$ between dots labeled by even integers.

By a similar argument, $\mathrm{inv}(\phi^{o})$ is equal to the number of 
inversions of $C$ between dots labeled by odd integers.

Summarizing the above discussions, we have Eqn. (\ref{invDyck}).
\end{proof}

\subsection{Dyck paths for Dellac configurations}
In this subsection, we study the relation between Dyck paths 
obtained from $\phi(C)$ and Dyck paths constructed in \cite{Big14}.
Dyck paths studied in \cite{Big14} have a weight on down steps.
Here, we focus on only Dyck paths and ignore its weights.

We follow \cite{Big14} for the definition of Dyck paths.
Let $C\in\mathrm{DC}_{1,2,n}$ be a Dellac configuration.
We introduce a Dyck path corresponding to $C$ by the following 
algorithm.
Let $j\in[1,n]$, and $e(i_{1}(j))$ and $e(i_{2}(j))$ with 
$i_{1}(j)<i_{2}(j)$ be two dots in the $j$-th column of $C$.
The values $i_1(j)$ and $i_2(j)$ are the numbers of rows enumerated 
from the bottom.

Suppose that $i_1(j)\le n<i_2(j)$ and $e(i,j)$ be a dot in the $i$-th row
and $j$-th column.
We define non-negative integers $l_C^{e}(e_1)$ and $r_C^{o}(e_{2})$
with $e_1=e(i_1(j))$ and $e_2=e(i_2(j))$ by 
\begin{align*}
l_{C}^{e}(e_1)&:=\#\{ e(i',j') : i_{1}(j)<i'\le n, j'<j \}, \\
r_{C}^{o}(e_2)&:=\#\{ e(i',j') : n<i'<i_2(j), j<j' \}.
\end{align*}
Recall that a Dyck path $p:=(p_1,p_2,\ldots, p_{2n})$ of length $2n$ consists of 
up steps and down steps. 
We denote an up step by $U$ and a down step by $D$.
Thus, each step $p_i$ is either $U$ or $D$.
We define a Dyck path for $C$ as follows. 
\begin{enumerate}
\item If $i_{2}(j)\le n$, we define $(p_{2j-1},p_{2j}):=(U,U)$.
\item If $i_1(j)\le n<i_2(j)$, we have two cases:
\begin{enumerate}
\item if $l_{C}^{e}(e_1)>r_{C}^{o}(e_2)$, then we define $(p_{2j-1},p_{2j}):=(D,U)$,
\item if $l_{C}^{e}(e_1)\le r_{C}^{o}(e_2)$, then we define $(p_{2j-1},p_{2j}):=(U,D)$.
\end{enumerate}
\item If $n<i_{1}(j)$, we define $(P_{2j-1},p_{2j}):=(D,D)$.
\end{enumerate}
We denote by $p(C)$ the path obtained from $C$ by the above algorithm.

\begin{prop}[Proposition 3.5 in \cite{Big14}]
Let $p(C)$ be defined as above. Then, $p(C)$ is a Dyck path of length $2n$.
\end{prop}

Let $C$ be a Dellac configuration in $\mathrm{DC}_{1,2,n}$.
We define the set $\mathfrak{U}(C)$ consisting of $n+1$ positive integers:
\begin{enumerate}
\item $2\in\mathfrak{U}(C)$, 
\item If $i_{2}(j)\le n$, the labels of two dots $e_1$ and $e_{2}$ are in $\mathfrak{U}(C)$,
\item If $i_1(j)\le n<i_2(j)$, we have two cases:
\begin{enumerate}
\item if $l_{C}^{e}(e_1)>r_{C}^{o}(e_2)$, then the label of $e_{2}$ is in $\mathfrak{U}(C)$,
\item $l_{C}^{e}(e_1)\le r_{C}^{o}(e_2)$, then the label of $e_{1}$ is in $\mathfrak{U}(C)$.
\end{enumerate}
\end{enumerate}
Note that if $i_1(j)>n$, then the labels of $e_{k}$ for $k=1,2$ are not in $\mathfrak{U}(C)$.

Let $\pi:=\phi(C)$ permutation define in Section \ref{def:phi}.
We define a subsequence of $\pi$ by
\begin{align*}
\pi':=\{\pi(i) : i\in\mathfrak{U}(C)\}.
\end{align*} 
Then, an integer sequence $\pi_{<}$ of length $n+1$ is defined as 
a unique increasing sequence of $\pi'$.

The following proposition is clear from the definitions of $\phi$ and $p(C)$.
\begin{prop}
Set $C\in\mathrm{DC}_{1,2,n}$.
Let $\pi_{<}$ be an increasing sequence obtained from $\phi(C)$ and 
$p(C)$ be a Dyck path as above.
we denote by $\overline{\mathcal{D}(\pi_{<})}$ a Dyck path obtained 
from $\mathcal{D}(\pi_{<})$ by deleting the first and last steps.
Then, we have 
\begin{align*}
\overline{\mathcal{D}(\pi_{<})}=p(C).
\end{align*}
\end{prop}

\subsection{Properties of \texorpdfstring{$\phi(C)$}{phi(C)}}
Given a triplet $(l,m,n)$, let $r:=|w_1|$ be the length of the word $w_1$ and 
$L$ be a total length of $\phi(C)$ defined in Section \ref{sec:GD}.
Let $\phi^{\mathrm{ref}}:=\phi^{\mathrm{ref}}_{1}\ldots\phi^{\mathrm{ref}}_{L/2}$
be an integer sequence 
such that 
\begin{eqnarray*}
\phi^{\mathrm{ref}}_{i}
:=
\begin{cases}
i, & 1\le i\le r, \\
r+m(i-r-1)+1, & r+1\le i\le L/2.
\end{cases}
\end{eqnarray*}

Fix $0\le p\le ln$. 
Then, $m$ integers $\{r+mp+j : 1\le j\le m\}$ are said to be 
in the same block and integers $r+mp+j$ and $r+mp'+k$ with 
$p\neq p'$ and $1\le j,k\le m$ are said to be in a different 
block.

\begin{remark}
\label{rmk:phi}
The word $\phi_{1}(C)$ is a concatenation of three words $w_{1}$, 
$\mathrm{word}(C)$ and $w_{2}$.
Therefore, the first $r=|w_1|$ letters in $\phi^{e}$ satisfies 
$\phi^{e}_{i}=i$ for $1\le i\le r$.
Further, we have $l$ dots in the first row and they are in from the 
first to the $l$-th columns.
This implies that $\phi_{i}=\phi^{\mathrm{ref}}_{i}$ 
for $r+1\le i\le r+l$. 
\end{remark}

\begin{theorem}
\label{thrm:charphi}
Let $r$ and $L$ be integers as above.
A permutation $\phi(C)\in\mathfrak{S}_{L}$ for a generalized 
configuration $C\in\mathrm{DC}_{l,m,n}$ 
satisfies the following conditions:
\begin{enumerate}
\item $\phi^{e}_{i}=\phi^{\mathrm{ref}}_{i}$ for $1\le i\le r+l$,
\item $\phi^{e}_{i}\le\phi^{\mathrm{ref}}_{i}$ for $r+l+1\le i\le L/2$,
\item integers in the same block appears in $\phi^{e}$ as an increasing sequence,
\item two integers $\phi^{e}_{i}$ and $\phi^{e}_{j}$ satisfying
$\lceil(i-r)/l\rceil=\lceil(j-r)/l\rceil$ are in a different block,
\item for a fixed $p\ge0$, $l$ integers $\phi^{e}_{i}$'s satisfying 
$\lceil(i-r)/l\rceil=p$ form an increasing sequence,
\item if $p-1$ does not appear in $\phi^{e}$ and $p$ appears in $\phi^{e}$, 
then $p\equiv r+1{\pmod m}$. 
\end{enumerate}
\end{theorem}
\begin{proof}
The property (1) is obvious from Remark \ref{rmk:phi}.

Recall that $\tau:=\tau_{C}$ is a permutation of length $lmn$ 
defined in Section \ref{sec:GDCalt}.
Since $\tau=\tau_{1}\ldots \tau_{lmn}$ is of length $lmn$, we divide it into 
two pieces: $\tau^{e}:=\tau_{1}\ldots\tau_{lmn/2}$ and $\tau^{o}:=\tau_{lmn/2+1}\ldots\tau_{lmn}$
for $mn$ even and $\tau^{e}:=\tau_{1}\ldots\tau_{l(mn-1)/2}$ and 
$\tau^{o}:=\tau_{l(mn-1)/2+1}\ldots\tau_{lmn}$
for $mn$ odd.
Then one can construct a sequence of integers $\phi^{e}$ and $\phi^{o}$ 
from $\tau$ as follows.
First, we determine $\phi^{e}_{i}=i$ for $1\le i\le r$, and 
$\phi^{e}_{r+j}=r+\tau^{e}_{j}$ for $1\le j\le L/2-r$. 
The odd part $\phi^{o}$ is given by 
$\phi^{o}_{i}=\tau^{o}_{i}$ where $1\le i\le lmn/2$ for $mn$ even 
and $1\le i\le l(mn+1)/2$ for $mn$ odd, 
$\phi^{o}_{j}=L/2+j$ where $lmn/2+1\le j\le L/2$ for $mn$ even,
and $\phi^{o}_{j}=L/2-l+j$ where $l(mn+1)/2+1\le j\le L/2$ for $mn$ odd.

For (2), recall that $\tau$ is a reading word of labels of dots $d_{i}$, $1\le i\le lmn$, 
from left to right and from bottom to top. 
Further, a configuration in $\mathrm{DC}_{l,m,n}$ contains $m$ dots in a single column,
the element $\tau_{i}$ satisfies $\tau_{i}\le m(i-1)+1$ for $1\le i$.
If we translate this condition in terms of $\phi^{e}$ satisfies the condition (2).

For (3), if two integers $p$ and $q$, $p<q$, are in the same block, 
two dots $d_{p}$ and $d_{q}$ are in the same column in the configuration $C$,
and the label of $d_{p}$ is strictly smaller than $d_{q}$.
By construction of $\tau$, two integers $p$ and $q$ appears as an increasing 
sequence in $\tau$, and so does $\phi^{e}$.
 
For (4), let $u:=\phi^{e}_{i}$ and $v:=\phi^{e}_{j}$ satisfying 
$\lceil(i-r)/l\rceil=\lceil(j-r)/l\rceil$. 
The condition $\lceil(i-r)/l\rceil=\lceil(j-r)/l\rceil$ means that 
two dots $d_{u}$ and $d_{v}$ have the same label. 
A configuration $C$ contains $l$ dots in a row. 
These $l$ dots have the same label. 
By the definition of a block, the dots are in a different block.
Thus, $u$ and $v$ are in a different block.

For (5), recall that when we construct $\tau$, we read a label 
of dots in a row from left to right.
Further, labels $i$'s of a dot $d_{i}$'s are increasing from 
left to right.  
This means that $\tau_{i}$'s satisfying $\lceil i/l\rceil=p$ 
form an increasing sequence, and (5) holds true.

For (6), suppose that $p-1$ does not appear in $\phi^{e}$ and 
$p$ appears in $\phi^{e}$.
Labels of a configuration $C$ are divided into two types according 
to their parities: 
labels in the upper half of $C$ are odd and labels in the 
lower half are even.
The assumption that $p$ appears in $\phi^{e}$ implies that 
if an integer $a$ is in the same block and $a<p$, then 
$a$ also appears in $\phi^{e}$. 
Thus, if the above assumption holds, $p$ should be the smallest 
integer in a block.
The configuration $C$ contains $m$ dots in a column, 
$p$ is always written as $p=r+qm+1$ with some non-negative integer $q$.
Thus, the condition (6) holds.
\end{proof}

Similarly, $\phi^{o}$ satisfies the similar properties as in 
Theorem \ref{thrm:charphi}.
We denote by $\overline{\phi^{o}}$ a word obtained from $\phi^{o}$
by $\overline{\phi^{o}}_{i}=L+1-\phi^{o}_{i}$ for all $1\le i\le L/2$.
We define $r':=|w_2|$.
\begin{theorem}
A word $\overline{\phi^{o}}$ satisfies the same conditions as $\phi^{e}$ 
in Theorem \ref{thrm:charphi} with $r'$ defined above.
Then, Theorem \ref{thrm:charphi} holds for $\overline{\phi^{o}}$.
\end{theorem}
\begin{proof}
A tableau obtained by rotating a Dellac configuration $180$ degrees 
is also a Dellac configuration.
Thus, the role of $\phi^{e}$ is replaced with $\overline{\phi^{o}}$, 
which implies Theorem holds true. 
\end{proof}

\section{Properties of generalized Dellac configurations}
\subsection{A map from \texorpdfstring{$\mathrm{DC}_{l,m,n}$}{DC(l,m,n)} 
to \texorpdfstring{$\mathrm{DC}_{1,2,l(m-1)n}$}{DC(1,2,l(m-1)n)}}
\label{sec:embed}
In Section \ref{sec:embed}, we embed a generalized Dellac configuration
in $\mathrm{DC}_{l,m,n}$ into $\mathrm{DC}_{1,2,n'}$ for some 
$n'$. 
We consider an embed which preserves the number of inversions.
For this purpose, we introduce two maps $\xi_1$ and $\xi_2$.

We define two maps $\xi_{1}:\mathrm{DC}_{l,m,n}\hookrightarrow\mathrm{DC}_{1,m,ln}$ 
and $\xi_{2}:\mathrm{DC}_{1,m,n}\hookrightarrow\mathrm{DC}_{1,2,(m-1)n}$.
By composing two maps $\xi_{1}$ and $\xi_{2}$, we embed a generalized 
Dellac configuration in $\mathrm{DC}_{l,m,n}$ into $\mathrm{DC}_{1,2,l(m-1)n}$.
 
\paragraph{\bf A map $\xi_{1}$}
Let $C\in\mathrm{DC}_{l,m,n}$. 
Then, by definition, each row in $C$ contains $l$ dots.
Recall the definition of the enumeration of dots in Section \ref{sec:GDCalt}.
Suppose a dot $f_{i}$ is a box whose Cartesian coordinate is $(x_{i},y_{i})$. 
Since we embed $C$ into $\mathrm{DC}_{1,m,nl}$, we consider a tableau $D$ of 
width $ln$ and height $mn$.
A dot $f'_{i}$ in the image $\xi_{1}(C)$ corresponds to a dot $f_{i}$ in $C$ 
as follows: 
the Cartesian coordinate of $f'_{i}$ is $(x_{i},i)$ in $D$.
The coordinate $x(i):=x_{i}$ of $f'_{i}$ satisfies  
\begin{eqnarray}
\label{eqn:xi1}
x(pl+1)<x(pl+2)<\ldots<x(pl+{l}), 
\end{eqnarray}
for $0\le p\le mn$.

Let $\overline{\mathrm{DC}}_{1,m,ln}\subset\mathrm{DC}_{1,m,ln}$ be the 
set of generalized Dellac configurations satisfying Eqn. (\ref{eqn:xi1}).
Then, one can easily construct the inverse 
$\xi^{-1}_{1}:\overline{\mathrm{DC}}_{1,m,ln}\rightarrow\mathrm{DC}_{l,m,n}$
by reversing the above procedure.
Thus, the map $\xi_{1}$ is a bijection between $\mathrm{DC}_{l,m,n}$ 
and $\overline{\mathrm{DC}}_{1,m,ln}$.

The map $\xi_{1}$ preserves the number of inversions, {\it i.e.}, 
\begin{eqnarray*}
\mathrm{inv}(C)=\mathrm{inv}(\xi_{1}(C)).
\end{eqnarray*}

\begin{example}
Let $C$ be a generalized Dellac configuration in $\mathrm{DC}_{2,2,2}$ (see 
left figure in Fig. \ref{fig:GCDxi1}).
Then, $\xi_{1}(C)$ is given by the right figure in Fig. \ref{fig:GCDxi1}.
We have $\mathrm{inv}(C)=\mathrm{inv}(\xi_{1}(C))=4$.
\begin{figure}[ht]
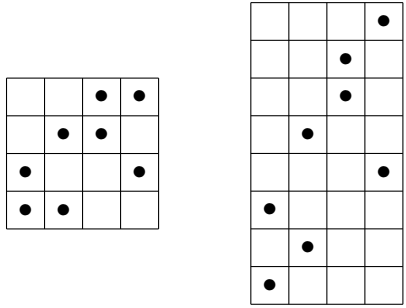

\tikzpic{-0.5}{
\draw(0,0)--(0,2)(1/2,0)--(1/2,2)(1,0)--(1,2)(3/2,0)--(3/2,2)(2,0)--(2,2);
\draw(0,0)--(2,0)(0,1/2)--(2,1/2)(0,1)--(2,1)(0,3/2)--(2,3/2)(0,2)--(2,2);
\draw(1/4,1/4)node{$\bullet$}(1/4,3/4)node{$\bullet$}
(3/4,1/4)node{$\bullet$}(3/4,5/4)node{$\bullet$}
(5/4,5/4)node{$\bullet$}(5/4,7/4)node{$\bullet$}
(7/4,3/4)node{$\bullet$}(7/4,7/4)node{$\bullet$};
}\qquad
\tikzpic{-0.5}{
\draw(0,0)--(2,0)(0,1/2)--(2,1/2)(0,1)--(2,1)(0,3/2)--(2,3/2)(0,2)--(2,2)(0,5/2)--(2,5/2)
(0,3)--(2,3)(0,7/2)--(2,7/2)(0,4)--(2,4);
\draw(0,0)--(0,4)(1/2,0)--(1/2,4)(1,0)--(1,4)(3/2,0)--(3/2,4)(2,0)--(2,4);
\draw(1/4,1/4)node{$\bullet$}(1/4,5/4)node{$\bullet$}
(3/4,3/4)node{$\bullet$}(3/4,9/4)node{$\bullet$}
(5/4,11/4)node{$\bullet$}(5/4,13/4)node{$\bullet$}
(7/4,7/4)node{$\bullet$}(7/4,15/4)node{$\bullet$};
}
\caption{A generalized Dellac configuration $C$ and $\xi_{1}(C)$}
\label{fig:GCDxi1}
\end{figure}
\end{example}

\paragraph{\bf A map $\xi_{2}$}
Let $C\in\mathrm{DC}_{1,m,n}$ and $C'\in\mathrm{DC}_{1,2,(m-1)n}$. 
Recall that a permutation $\tau_{C}\in\mathfrak{S}_{mn}$ characterizes
a Dellac configuration $C$ (see Section \ref{sec:GDCalt}).
We will construct a map $\xi_{2}$ in terms of $\tau_{C}$ by defining 
a map from $\tau_{C}$ to $\tau_{C'}$.
Let $A\subset\{1,2,\ldots,lm\}$ be a set of positive integers 
such that $r\in A$ if and only if $r\not\equiv0,1\pmod{m}$.

Let $a\in A$ be the minimum element in $A$.
Since an integer $\{1,2,\ldots,mn\}$ appears exactly once in 
the permutation $\tau_{C}$, we put one more $a$ soon after 
the integer $a$ in $\tau_{C}$.
We attach subscripts $1$ and $2$ to two $a$'s from left to right 
and make them distinct.  
We denote by $a_{i}$ the integer $a$ with a subscript $i$. 
We define the order of two $a$'s by $a_{1}<a_{2}$.
Let $\overline{\tau}_{C}:=\overline{\tau}_{1}\ldots\overline{\tau}_{mn+1}$ 
be a word obtained from $\tau_{C}$ by the above procedure.
Let $r$ be a positive integer such that $\overline{\tau}_{r}=a_{1}$ 
and $\overline{\tau}_{r+1}=a_{2}$.
We define two sets by
\begin{eqnarray*}
\mathrm{InvL}(a)&:=&\{s: \overline{\tau}_{s}>a, 1\le s\le r-1 \}, \\
\mathrm{InvR}(a)&:=&\{s: \overline{\tau}_{s}<a, r+2\le s\le mn+1 \}. 
\end{eqnarray*}
We define a pair of non-negative integers 
by $v_{a}:=(\#\mathrm{InvL}(a),\#\mathrm{InvR}(a))$.
We denote by $V_{A}:=\{v_{a}: a\in A\}$.

Let $b(1)<b(2)<\ldots<b(t)$ be a increasing sequence such that 
$b(i)\in\mathrm{InvL}(a)$ and $t=\#\mathrm{InvL}(a)$.
Similarly, let $c(1)<c(2)<\ldots<c(u)$ be a increasing sequence such that 
$c(i)\in\mathrm{InvL}(a)$ and $u=\#\mathrm{InvR}(a)$.
We perform the following operation on $\overline{\tau}_{C}$ and 
obtain a new word $\overline{\tau}'_{C}:=\overline{\tau}'_{1}\ldots\overline{\tau}'_{mn+1}$.
If $s\notin\mathrm{InvL}(a)\cup\mathrm{InvR}(a)$, we have $\overline{\tau}_{s}=\overline{\tau}'_{s}$.
We define $\overline{\tau}'_{b(i)}=\overline{\tau}_{b(i-1)}$ for $2\le i\le t$, 
$\overline{\tau}'_{r+1}=\overline{\tau}_{b(t)}$, and 
$\overline{\tau}'_{b(1)}=a_{2}$. 
Similarly, we define $\overline{\tau}'_{c(i)}=\overline{\tau}_{c(i+1)}$ for $1\le i\le u-1$, 
$\overline{\tau}'_{r}=\overline{\tau}_{c(1)}$, and 
$\overline{\tau}'_{c(u)}=a_{1}$.

Let $a'$ be the minimum element in $A\setminus\{a\}$.
We perform the same procedure as $a$ on $\overline{\tau}'_{C}$ 
and obtain a new word of length $mn+2$.
We continue this procedure for all the elements in $A$ and obtain
a word of length $2(m-1)n$.
Finally, we standardize the obtained word according to the total 
order of alphabets and have $\tau_{C'}$. 

It is obvious that the map $\xi_{2}$ is an injection.
Thus, the action of $\xi_{2}^{-1}$ on a Dellac configuration 
in $\mathrm{DC}_{1,2,(m-1)n}$ may not give a configuration in $\mathrm{DC}_{1,m,n}$. 	 
We have a bijection between $\tau_{C}$ and a pair $(\tau_{C'},V_{A})$ if and only if 
$(\tau_{C'},V_{A})$ is admissible.
One can construct the inverse map by reversing the above map.
Through a vector $v_{a}\in V_{A}$, we specify the position of $a_{1}$ and $a_{2}$
in $\tau_{C'}$.

By construction, the map $\xi_{2}$ preserves the number of inversions:
\begin{eqnarray*}
\mathrm{inv}(C)=\mathrm{inv}(\xi_{2}(C)).
\end{eqnarray*}

\begin{remark}
The existance of $V_{A}$ is useful when we construct the inverse map $\xi_{2}^{-1}$.
For example, we have
\begin{eqnarray*}
8_{2}2\underline{11}3\underline{12}48_{1}\rightarrow
\begin{cases} 
\underline{11}2\underline{12}8_{1}8_{2}34, & \text{for } v_{8}=(2,2), \\  
\underline{11}8_{1}8_{2}2\underline{12}34, & \text{for } v_{8}=(1,3).
\end{cases}
\end{eqnarray*}
Note that in both cases, the numbers of inversions are the same.
The map $\xi_{2}$ is an injection, at most one of the above examples gives
a certain generalized Dellac configuration.
\end{remark}

\begin{example}
Let $C$ be a generalized Dellac configuration in $\mathrm{DC}_{1,3,3}$ (see the left 
figure in Fig.~\ref{fig:GCDxi2}).
Then we have 
\begin{eqnarray*}
\tau_{C}=142756389 &\rightarrow& 142_{1}2_{2}756389\rightarrow
12_{2}2_{1}4756389  \\
&\rightarrow& 12_{2}2_{1}475_{1}5_{2}6389 
\rightarrow 12_{2}2_{1}45_{2}3765_{1}89 \\
&\rightarrow&12_{2}2_{1}45_{2}3765_{1}8_{1}8_{2}9 \\
&\rightarrow&132574986\underline{10}\underline{11}\underline{12}=\tau_{C'}.
\end{eqnarray*}
The Dellac configuration $\xi_{2}(C)$ corresponding to $\tau_{C'}$ 
is given by the right figure in Fig. \ref{fig:GCDxi2}.
We have $\mathrm{inv}(C)=\mathrm{inv}(C')=7$.
\begin{figure}[ht]
\tikzpic{-0.5}{
\draw(0,0)--(3/2,0)(0,1/2)--(3/2,1/2)(0,1)--(3/2,1)(0,3/2)--(3/2,3/2)
(0,2)--(3/2,2)(0,5/2)--(3/2,5/2)(0,3)--(3/2,3)(0,7/2)--(3/2,7/2)
(0,4)--(3/2,4)(0,9/2)--(3/2,9/2);
\draw(0,0)--(0,9/2)(1/2,0)--(1/2,9/2)(1,0)--(1,9/2)(3/2,0)--(3/2,9/2);
\draw(1/4,1/4)node{$\bullet$}(1/4,5/4)node{$\bullet$}(1/4,5/4)node{$\bullet$}
(1/4,13/4)node{$\bullet$};
\draw(3/4,3/4)node{$\bullet$}(3/4,9/4)node{$\bullet$}(3/4,11/4)node{$\bullet$};
\draw(5/4,7/4)node{$\bullet$}(5/4,15/4)node{$\bullet$}(5/4,17/4)node{$\bullet$};
}
\qquad
\tikzpic{-0.5}{
\draw(0,0)--(3,0)(0,1/2)--(3,1/2)(0,1)--(3,1)(0,3/2)--(3,3/2)(0,2)--(3,2)
(0,5/2)--(3,5/2)(0,3)--(3,3)(0,7/2)--(3,7/2)(0,4)--(3,4)(0,9/2)--(3,9/2)
(0,5)--(3,5)(0,11/2)--(3,11/2)(0,6)--(3,6);
\draw(0,0)--(0,6)(1/2,0)--(1/2,6)(1,0)--(1,6)(3/2,0)--(3/2,6)(2,0)--(2,6)
(5/2,0)--(5/2,6)(3,0)--(3,6);
\draw(1/4,1/4)node{$\bullet$}(1/4,5/4)node{$\bullet$}
(3/4,3/4)node{$\bullet$}(3/4,11/4)node{$\bullet$}
(5/4,7/4)node{$\bullet$}(5/4,17/4)node{$\bullet$}
(7/4,9/4)node{$\bullet$}(7/4,15/4)node{$\bullet$}
(9/4,13/4)node{$\bullet$}(9/4,19/4)node{$\bullet$}
(11/4,21/4)node{$\bullet$}(11/4,23/4)node{$\bullet$};
}
\caption{A generalized Dellac configuration $C\in\mathrm{DC}_{1,3,3}$ 
and $C'=\xi_{2}(C)\in\mathrm{DC}_{1,2,6}$}
\label{fig:GCDxi2}
\end{figure}
\end{example}

\subsection{Sets bijective to generalized Dellac configurations}
\label{sec:tuple1}
In this subsection, we introduce a description of generalized 
Dellac configurations in terms of sets.
In Propostion 3.1 in \cite{Fei11}, tuples $I^1,\ldots,I^n-1$ with 
$I^{l}\subset[1,n]$ and $\#(I^{l})=l$ are introduced 
to show the definitions of Genocchi numbers by Dellac and Kreweras 
are equivalent.
We generalize this description to the case of general $(l,m,n)$.
We also introduce tuples as in the case of $l=1$, and allow that 
a tuple contains the same integers.

Fix a triplet $(l,m,n)$.
We write an integer $i\in[1,ln]$ as $i:=pl+q$ with $0\le p\le n-1$ and $1\le q\le l$.
\begin{defn}
\label{defn:I}
We define a collection $\mathbf{I}$ of $ln$ pairs of a tuple and a sequence of 
non-negative integers 
$\mathbf{I}:=\{(I^{0},J^{0}),\ldots,(I^{ln-1},J^{ln-1})\}$ 
satisfying the following conditions: 
\begin{enumerate}
\item $I^{i}\subset \{1^{l}, 2^{l}\dots,((m-1)n)^{l}\}$ and $\#I^{i}=(m-1)i$, 
\item $J^{i}:=(J^{i}_{1},\ldots,J^{i}_{(m-1)n})\in [0,l+1]^{(m-1)n}$, $q\le J^{i}_{p+1}\le l$, 
and $J^{i}_r=J^{i-1}_r$ or $J^{i-1}_{r}+1$.
\item $I^{0}:=\emptyset$ and $J^{0}:=(0,\ldots,0)$,
\item $I^{i-1}\setminus\{p+1\}\subset I^{i}$ for $1\le i\le l(n-1)$ and 
$I^{i-1}\subset I^{i}$ for $l(n-1)+1\le i\le ln-1$,
\item If $p\le n-2$, $p+1\notin I^{i-1}$ and $J^{i-1}_{p+1}\ge q$, then
\begin{enumerate}
\item $J_{p+1}^{i}=J^{i-1}_{p+1}$, $p+1\notin I^{i}$ and 
$\#(I^{i}\setminus I^{i-1})=m-1$, or 
\item $J_{p+1}^{i}=J^{i-1}_{p+1}+1$, $\#(I^{i}\setminus I^{i-1})=m-1$,
\end{enumerate} 
\item If $p\le n-2$, $p+1\notin I^{i-1}$ and $J^{i-1}_{p+1}=q-1$, then 
$J_{p+1}^{i}=q$ and $\#(I^{i}\setminus I^{i-1})=m-1$,
\item If $p\le n-2$, $p+1\in I^{i-1}$ and $J^{i-1}_{p+1}\ge q$,
\begin{enumerate}
\item $J^{i}_{p+1}=J^{i-1}_{p+1}$, $p+1\notin I^{i}\setminus I^{i-1}$ and 
$\#(I^{i}\setminus I^{i-1})=m-1$, or
\item $J^{i}_{p+1}=J^{i-1}_{p+1}$, $p+1\notin I^{i}\setminus (I^{i-1}\setminus \{p+1\})$ and 
$\#(I^{i}\setminus (I^{i-1}\setminus \{p+1\}))=m$, or
\item $J^{i}_{p+1}=J^{i-1}_{p+1}+1$ and $\#(I^{i}\setminus I^{i-1})=m-1$,
\end{enumerate}
\item If $p\le n-2$, $p+1\in I^{i-1}$ and $J^{i-1}_{p+1}=q-1$, then 
$J^{i}_{p+1}=q$ and $\#(I^{i}\setminus I^{i-1})=m-1$,
\item Let $\widetilde{I^i}:=I^{i}\setminus I^{i-1}$ for cases $p=n-1$ or cases (5) to (8) except (7b).
Similarly, $\widetilde{I^i}:=I^{i}\setminus (I^{i-1}\setminus \{p+1\})$ for case (7b).
The elements in $\widetilde{I^i}$ are all distinct. 
If $r\in\widetilde{I^i}$ with $1\le r\le n$, $J^{i}_{r}=J^{i-1}_{r}+1$.
If $p+1\neq r\notin\widetilde{I^i}$ with $1\le r\le n$, $J^{i}_{r}=J^{i-1}_{r}$.
\end{enumerate}
\end{defn}

\begin{prop}
\label{prop:IGCD}
The number of collections $\mathbf{I}$ satisfying the conditions 
in Definition \ref{defn:I} is equal to the number of generalized Dellac configurations 
in $\mathrm{DC}_{l,m,n}$.
\end{prop}

\begin{proof}

For $i=pl+q$ and $1\le k\le mn$, we set $[k(i)]_{+}=k$ if $k>p+1$ and 
$[k(i)]_{+}=k+(m-1)n$ if $k\le p+1$.
Similarly, we set $[k(i)]_{-}=k$ if $p+1<k\le (m-1)n$ and 
$[k(i)]_{-}=k-(m-1)n$ for $(m-1)n+1\le k$.

Given a collection $\mathbf{I}$, we will construct the corresponding generalized 
Dellac configuration $D$ and show that the map is one-to-one.
The positions of dots of $D$ in the $i$-th column are given 
by the following rules.

We consider the case where $p\le n-2$.
First, suppose $p+1\notin I^{i-1}$ and $J_{p+1}^{i-1}\ge q$.
Then, because of the condition (5), we have two cases (5a) and (5b).
For (5a), we have $p+1\notin I^{i}$ and $I^{i}\setminus I^{i-1}$ contains 
exactly $m-1$ integers $j_1,\ldots,j_{m-1}$ which are all distinct 
by the condition (9). 
Then, $D$ contains the dots whose Cartesian coordinates are 
$(i,[j_{k}(i)]_{+})$, $1\le k\le m-1$ and $(i,p+(m-1)n+1)$.
For case (5b), the set $I^{i}\setminus I^{i-1}$ contains exactly 
$m-1$ integers $j_1,\ldots,j_{m-1}$ which are all distinct. 
Note that $I^{i}\setminus I^{i-1}$ may or may not contain the integer $p+1$.
Then, $D$ contains the dots whose Cartesian coordinates are 
$(i,[j_{k}(i)]_{+})$, $1\le k\le m-1$ and $(i,p+1)$.

Secondly, suppose $p+1\notin I^{i-1}$ and $J_{p+1}^{i-1}=q-1$.
The difference $I^{i}\setminus I^{i-1}$ contains exactly 
$m-1$ distinct integers $j_{1},\ldots,j_{m-1}$. 
Then, $D$ contains the dots whose Cartesian coordinates 
are $(i,[j_{k}(i)]_{+})$, $1\le k\le m-1$ and $(i,p+1)$. 

Thirdly, suppose $p+1\in I^{i-1}$ and $J_{p+1}^{i-1}\ge q$.
We have three cases (7a), (7b) and (7c).
For (7a), the difference $I^{i}\setminus I^{i-1}$ contains exactly $m-1$
distinct integers $j_{1},\ldots,j_{m-1}$. 
Note that the $I^{i}\setminus I^{i-1}$ may or may not contain 
the integer $p+1$.
Then, $D$ contains the dots whose Cartesian coordinates are 
$(i,[j_{k}(i)]_{+})$, $1\le k\le m-1$ and $(i,p+(m-1)n+1)$.
For (7b), the difference $I^{i}\setminus (I^{i-1}\setminus \{p+1\})$ 
contains exactly $m$ distinct elements $j_{1},\ldots,j_{m}$ and 
does not contain the integer $p+1$. 
Then, $D$ contains the dots whose Cartesian coordinates are 
$(i,[j_{k}(i)]_{+})$, $1\le k\le m$.
For (7c), the positions of dots are the same as the case (5b).

Finally, suppose $p+1\in I^{i-1}$ and $J_{p+1}^{i-1}=q-1$.
The positions of dots in $D$ are the same as the case (5b), which 
implies the condition (8).

We consider the case with $p=n-1$. 
From condition (4), we have $I^{i-1}\subset I^{i}$ and 
the difference $I^{i}\setminus I^{i-1}$ contains exactly 
$m-1$ distinct integers $j_{1},\ldots,j_{m-1}$.
Then, $D$ contains the dots whose Cartesian coordinates are 
$(i,[j_{k}(i)]_{+})$, $1\le k\le m-1$ and $(i,mn)$.

For the $ln$-th column, we put $m$ dots in the unique way such that 
$D$ is a generalized Dellac configuration.

To show that this map is a bijection, we will construct the inverse 
map.
Let $D$ be a generalized Dellac configuration.
Given $D$, we construct the pair $(I^{i},J^{i})$ inductively.
Let $i=1$ and suppose that the number of rows of dots in the first column 
of $D$ are $m$ distinct integers $1<j_{1}<\ldots<j_{m-1}$. 
If $j_{m-1}=(m-1)n+1$, then we define $I^{1}:=\{1,j_{1},\ldots,j_{m-2}\}$.  
Otherwise, we define $I^{1}:=\{j_{1},\ldots,j_{m-1}\}$.
In both cases, we define $J^{1}_{1}:=J^{0}_{1}+1=1$.
If an integer $1\le r\le n$ appears in $I^{1}$, we define 
$J^{1}_{r}:=J^{0}_{r}+1=1$.
Thus, the pair $(I^{1},J^{1})$ satisfies the conditions (2) and (9).

Assume that the pair $(I^{i-1},J^{i-1})$ for $p\le n-2$ is already defined.
First, suppose that the dot in $(i,p+1)$ belongs to $D$.
Since $\#I^{i}=(m-1)i$, we have $m-1$ dots in the boxes $(i,j_{k})$, 
$1\le k\le m-1$, in $D$ with $p+1\le j_{1}<j_{2}<\ldots<j_{m-1}\le (m-1)n+p+1$. 
From the condition (2) and the fact that we have $l$ dots in a row of $D$, 
we have $q-1\le J^{i-1}_{p+1}\le l-1$.
Then, we set $I^{i}=I^{i-1}\cup\{[j_{1}(i)]_{-},\ldots,[j_{m-1}(i)]_{-}\}$
and $J^{i}_{p+1}=J^{i-1}_{p+1}+1$.
If $r\in I^{i}\setminus I^{i-1}$ with $1\le r\le n$, we set 
$J^{i}_{r}=J^{i-1}_{r}+1$.
This corresponds to the cases (5b), (6), (7c) and (8).
Second, suppose that the $(i,p+1)$ box does not have a dot.
Then, the boxes $(i,j_{k})$, $1\le k\le m$, contain a dot in $D$ 
with $j_{1}<j_{2}<\ldots<j_{m}$.
We have two cases: (i) $[j_{m}(i)]_{-}=p+1$ and (ii) $[j_{m}(i)]_{-}\neq p+1$.
For case (i), we set $I^{i}=I^{i}\cup\{[j_{1}(i)]_{-},\ldots,[j_{m-1}(i)]_{-}\}$ 
and $J^{i}_{p+1}=J^{i}_{p+1}$. 
If $r\in I^{i}\setminus I^{i-1}$ with $1\le r\le n$, we set 
$J^{i}_{r}=J^{i-1}_{r}+1$.
This corresponds to the cases (5a) and (7a).
For case (ii), we have $J^{i-1}_{p+1}\ge q$ and $p+1\in I^{i-1}$ since we have 
at least $q$ dots in a row of $D$. 
Then, we set 
$I^{i}=(I^{i-1}\setminus\{p+1\})\cup\{[j_{1}(i)]_{-},\ldots,[j_{m}(i)]_{-}\}$
and $J^{i}_{p+1}=J^{i-1}_{p+1}$.
If $r\in I^{i}\setminus (I^{i-1}\setminus\{p+1\})$ with $1\le r\le n$, we set 
$J^{i}_{r}=J^{i-1}_{r}+1$.
This corresponds to the case (7b).

Assume that $(I^{i-1},J^{i-1})$ is already defined for $p=n-1$.
Then, we have a dot in the box $(i,mn)$ and dots in the boxes 
$(i,j_{k})$, $1\le k\le m-1$, with $n\le j_{1}<\ldots<j_{m-1}\le mn-1$.
We set $I^{i}=I^{i-1}\cup\{[j_{1}(i)]_{-},\ldots,[j_{m-1}(i)]_{-}\}$ and 
$J^{i}_{p+1}=J^{i-1}_{p+1}+1$.
This completes the proof.
\end{proof}

\subsection{Alternative description of sets bijective to generalized 
Dellac configurations}
\label{sec:tuple2}
In Definition \ref{defn:I}, we give a description of a collection 
$\mathbf{I}$ with repeated elements. 
In this subsection, we introduce another simple description of 
a collection $\mathbf{K}$ without repeated elements.

\begin{defn}
\label{defn:K}
We define a collection $\mathbf{K}:=(K^{1},\ldots,K^{ln-1})$ of 
tuples satisfying the following conditions:
\begin{enumerate}
\item $K^{i}\subset\{1,2,\ldots,l(m-1)n\}$ and $\#K^{i}=(m-1)i$, 
\item $K^{i-1}\setminus\{i\}\subset K^{i}$ for $2\le i\le ln-1$,
\item Suppose $r=pl+q$ is uniquely written by $p$ and $q$ 
with $0\le p\le n-1$ and $1\le q\le l$.
\begin{enumerate} 
\item 
If $r\in K^{i}$, then all $r':=pl+q'$ satisfying $1\le q'\le q$ are 
in $K^{i}$,
\item Let $r'=pl+q'$ with $1\le q'\neq q\le l$.
If $r\in K^{i}\setminus(K^{i-1}\setminus\{i\})$, then 
$r'\notin K^{i}\setminus(K^{i-1}\setminus\{i\})$.
\end{enumerate}
\end{enumerate}
\end{defn}

\begin{prop}
The number of collections $\mathbf{K}$ satisfying the conditions 
in Definition \ref{defn:K} is equal to the number of generalized 
Dellac configurations in $\mathrm{DC}_{l,m,n}$.
\end{prop}
\begin{proof}
We set $[[k(i)]]_{+}=k$ if $i+1\le k\le l(m-1)n$, and 
$[[k(i)]]_{+}=k+l(m-1)n$ if $1\le k\le i$. 
Similarly, we set $[[k(i)]]_{-}=k$ if $i<k\le (m-1)n$ and 
$[[k(i)]]_{-}=k-(m-1)n$ for $(m-1)n+1\le k$.

Let $C\in\mathrm{DC}_{l,m,n}$ and $D:=\xi_{1}(C)\in\mathrm{DC}_{1,m,ln}$.
Given a collection $\mathbf{K}$, we will construct the corresponding 
generalized Dellac configuration $D$ and show that the map 
is one-to-one.
Note that $\mathbf{K}$ does not have a repeated elements and 
this corresponds to the statement that there exists a single dot 
in a row of $D$.
The positions of dots of $D$ in the $i$-th column are given 
from a collection $\mathbf{K}$ by the following rules.

Suppose that $i\notin K^{i-1}$. 
Then, from the condition (2), the difference $I^{i}\setminus I^{i-1}$ 
contains exactly $m-1$ integer $j_{1},\ldots,j_{m-1}$.
The configuration $D$ contains dots in the boxes $(i,i)$ and 
$(i,[[j_{k}(i)]]_{+})$ for $1\le k\le m-1$.

Suppose $i\in K^{i-1}$.
Then, we have two cases: (i) $i\in K^{i}$ and (ii) $i\notin K^{i}$.
For (i), we have exactly $m-1$ elements $j_{1},\ldots,j_{m-1}$
in $K^{i}\setminus K^{i-1}$, where $j_{k}\neq i$.
Then, $D$ contains dots in the boxes $(i,[[i]]_{+})$ and  
$(i,[[j_{k}(i)]]_{+})$ for $1\le k\le m-1$.
For (ii), we have exactly $m$ elements $j_{1},\ldots,j_{m}$
in $K^{i}\setminus (K^{i-1}\setminus\{i\})$, where 
$j_{k}\neq i$.
Then, $D$ contains dots in the boxes $(i,[[j_{k}(i)]]_{+})$
for $1\le k\le m$.

It is easy to show that the conditions (3a) and (3b) are 
equivalent to the condition that $D\in\overline{\mathrm{DC}}_{1,m,ln}$, 
namely, the condition (\ref{eqn:xi1}).

We construct the inverse map to show that this map is a bijection.
Let $D$ be a generalized Dellac configuration in $\mathrm{DC}_{1,m,ln}$.
Given $D$, we construct $K^{i}$ inductively.
Let $i=1$ and suppose that the number of rows of dots in the first 
column of $D$ are $m$ distinct integers $1<j_{1}<\ldots<j_{m-1}$.
If $j_{m-1}=(m-1)n+1$, then we define $K^{1}:=\{1,j_{1},\ldots,j_{m-2}\}$.
Otherwise, we define $K^{1}:=\{j_{1},\ldots,j_{m-1}\}$.

Assume that $K^{i-1}$ for $1\le i$ is already defined.
First, suppose that a dot in the box $(i,i)$ belongs to $D$.
Since $\#I^{i}=(m-1)i$, we have $m-1$ dots in the boxes 
$(i,j_{k})$, $1\le k\le m-1$, in $D$ 
with $i<j_{1}<\ldots<j_{m-1}\le (m-1)n+i$.
Then, we define $I^{i}:=I^{i-1}\cup\{[[j_{1}(i)]]_{-},\ldots,[[j_{m-1}(i)]]_{-}\}$.
Second, suppose that there is no dot in the box $(i,i)$ in $D$.
Since the configuration $D$ contains a single dot in a row, there exists $i'<i$ 
such that the box $(i',i)$ contains a dot. 
Therefore, $i\in K^{i-1}$.
We have $m$ dots in the boxes $(i,j_{k})$, $1\le k\le m$, in $D$ such that 
$i<j_{1}<\ldots<j_{m}$.
Then, we define 
$K^{i}:=(K^{i-1}\setminus\{i\})\cup\{[[j_{1}(i)]]_{-},\ldots,[[j_{m}(i)]]_{-}\}$.
This completes the proof.
\end{proof}

\section{Dellac configurations with general boundaries}
\label{sec:DCgb}
In this section, we study a generalization of Dellac configurations 
with parameters $(l,m,n)=(1,2,n)$.
We denote by $\delta_{n}$ the staircase of size $n$, {\it i.e.},
$\delta_{n}=(n,n-1,\ldots,1)$.
Let $\lambda$ be a partition, {\it i.e.}, $\lambda:=(\lambda_1,\ldots,\lambda_l)$ 
satisfies $\lambda_1\ge\lambda_2\ge\ldots\ge\lambda_{l}\ge0$.
In this section, we consider only partitions $\lambda$ such that $l\le n-1$ and 
each integer appears at most twice in $\lambda$.

Let $\lambda$ and $\mu$ be partitions.
We denote by $\lambda\oplus\mu$ a partition obtained from $\lambda$ and $\mu$ 
by concatenating two partitions, {\it i.e.}, 
$\lambda\oplus\mu:=(\lambda_1,\ldots,\lambda_l,\mu_1,\ldots,\mu_{l'})$ 
where $l$ and $l'$ are the sizes of $\lambda$ and $\mu$.

Give a partition $\lambda$, we denote by $\lambda^{i}:=\lambda\setminus\{\lambda_{i}\}$ 
a partition obtained from $\lambda$ by deleting $\lambda_{i}$.

\subsection{Dellac configurations with general boundaries}
Recall a Dellac configuration is a configuration of $2n$ dots 
in the tableau with $2n$ rows and $n$ columns such that 
each column has two dots with certain conditions as in 
Definition \ref{defn:GDC}. 
The third condition in Definition \ref{defn:GDC} defines 
a region in which we are not allowed to put dots.
There are two such regions at both top and bottom.
The shape of the regions is a staircase of size $n-1$. 
By changing the staircase to the general shape, 
we naturally define a generalized Dellac configuration
with general boundaries.

Let $\lambda:=(\lambda_1,\ldots,\lambda_{n-1})$ 
and $\mu:=(\mu_1,\ldots,\mu_{n-1})$ be partitions 
inside the staircase $\delta_{n-1}$.
In the rectangle with $2n$ rows and $n$ columns, 
we denote by the $(i,j)$-box the box in the $i$-th row 
from top and the $j$-th column from left.
\begin{defn}
A generalized Dellac configuration of size $n$ with boundaries $\lambda$ and $\mu$
is a configuration of $2n$ dots such that it satisfies the first two condition in
Definition \ref{defn:GDC} and the following conditions:
\begin{enumerate}
\item there are no dots in the $(i,j)$-box satisfying $1\le i\le n-1$ and 
$1\le j\le \lambda_{i}$, 
\item there are no dots in the $(i,j)$-box satisfying $n+2\le i\le 2n$ and 
$n+1-\mu_{2n+1-i}\le j\le n$.	
\end{enumerate}
We denote by $C^{n}(\lambda,\mu)$ the total number of generalized 
Dellac configuration with boundaries $\lambda$ and $\mu$.
The top boundary is $\lambda$ and the boundary at the bottom is $\mu$.
\end{defn}

\begin{example}
Let $n=3$, $\lambda=(2)$ and $\mu=(2,1)$.
We have $9$ Dellac configurations, and $7$ of them are equivalent 
to Dellac configurations with $n=3$ and $\lambda=\mu=(2,1)$.
Two non-trivial Dellac configurations are shown in Figure \ref{fig:GDC3221}.
The shaded regions correspond to the boundary partitions. 
\begin{figure}[ht]
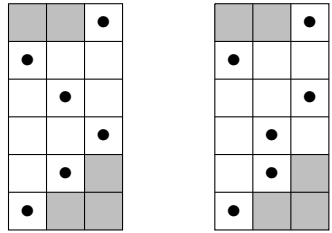

\tikzpic{-0.5}{[scale=0.5]
\fill[lightgray](0,0)--(2,0)--(2,-1)--(0,-1)--(0,0);
\fill[lightgray](1,-6)--(1,-5)--(2,-5)--(2,-4)--(3,-4)--(3,-6)--(1,-6);
\draw(0,0)--(3,0)--(3,-6)--(0,-6)--(0,0);
\draw(1,0)--(1,-6)(2,0)--(2,-6);
\draw(0,-1)--(3,-1)(0,-2)--(3,-2)(0,-3)--(3,-3)(0,-4)--(3,-4)(0,-5)--(3,-5);
\draw(2.5,-0.5)node{$\bullet$}(0.5,-5.5)node{$\bullet$}
(0.5,-1.5)node{$\bullet$}(1.5,-4.5)node{$\bullet$};
\draw(1.5,-2.5)node{$\bullet$}(2.5,-3.5)node{$\bullet$};
}\qquad
\tikzpic{-0.5}{[scale=0.5]
\fill[lightgray](0,0)--(2,0)--(2,-1)--(0,-1)--(0,0);
\fill[lightgray](1,-6)--(1,-5)--(2,-5)--(2,-4)--(3,-4)--(3,-6)--(1,-6);
\draw(0,0)--(3,0)--(3,-6)--(0,-6)--(0,0);
\draw(1,0)--(1,-6)(2,0)--(2,-6);
\draw(0,-1)--(3,-1)(0,-2)--(3,-2)(0,-3)--(3,-3)(0,-4)--(3,-4)(0,-5)--(3,-5);
\draw(2.5,-0.5)node{$\bullet$}(0.5,-5.5)node{$\bullet$}(0.5,-1.5)node{$\bullet$}
(1.5,-4.5)node{$\bullet$};
\draw(1.5,-3.5)node{$\bullet$}(2.5,-2.5)node{$\bullet$};
}
\caption{Two Dellac configurations with $n=3$, $\lambda=(2)$ and $\mu=(2,1)$.}
\label{fig:GDC3221}
\end{figure}
\end{example}

To construct a Dellac configuration, we place two dots in each column, and 
we are not allowed to put dots in the regions determined by boundary partitions 
$\lambda$ and $\mu$.
A row $r$ in a column is said to be admissible if $r$ is not in the boundary 
partitions.
Let $p_{i}$, $1\le i\le n$ be the number of admissible rows in the $i$-th 
column. 
Note that $p_i$ depends on the both boundaries $\lambda$ and $\mu$.
For example, we have $(p_1,p_2,p_3)=(5,4,4)$ when $n=3$, $\lambda=(1)$ 
and $\mu=(2,1)$. 
We define $q=\max\{p_i-n-1: 1\le i\le n\}$.

As in Section \ref{def:phi}, we put a label on a row as follows.
We put a label $2(i+q+1)$ on the $i$-th row from bottom,
and a label $2i-1$ on the $i+n$-th  row from bottom for $1\le i\le n$.
We define a permutation $\sigma(C)$ for a Dellac configuration $C$ by
\begin{align*}
\sigma(C):=w_1\ast \mathrm{word}(C)\ast w_2,
\end{align*}
where $\mathrm{word}(C)$ is similarly defined as in Section \ref{def:phi} and 
\begin{align*}
w_1&=(2,4,\ldots,2(q+1)) \\
w_2&=(2n+1,2n+3,\ldots,2n+2q+1).
\end{align*}
We set $L=n+q+1$. 
Given a permutation $\sigma:=\sigma(C)$, we define 
$\sigma^{o}:=\sigma(1)\sigma(3)\ldots\sigma(2L-1)$ and 
$\sigma^{e}:=\sigma(2)\sigma(4)\ldots\sigma(2L)$.

We define a statistics $\mathrm{st}(\sigma)$ as in Eqn. (\ref{eqn:defst}) 
by
\begin{align*}
\mathrm{st}(\sigma):=
L^{2}-\sum_{1\le i\le L}\sigma(2i)-\mathrm{inv}(\sigma^{o})-\mathrm{inv}(\sigma^{e}).
\end{align*}
We define $\mathrm{inv}(C)$ as the number of inversions of $C$ as in Section \ref{sec:GDC}.
By a similar argument to Theorem \ref{thrm:varphi}, we have the following proposition.
\begin{prop}
Let $C$ be a Dellac configuration with general boundaries.
Then, we have
\begin{align*}
\mathrm{st}(\sigma)=\genfrac{(}{)}{0pt}{}{L}{2}-\mathrm{inv}(C).
\end{align*}
\end{prop}

\begin{remark}
One can easily show by the same argument that 
Theorem \ref{thrm:Dyckinv} also holds for a Dellac configuration $C$ with general boundaries
for $\sigma(C)$.
\end{remark}

\subsection{Enumeration of Dellac configurations with general boundaries}
In this subsection, we study enumerations of Dellac configurations.
When $\mu$ is a staircase of size $n-1$, we abbreviate $C^{n}(\lambda,\delta_{n-1})$ 
as simply $C^{n}(\lambda)$.
Let $l:=l(\lambda)$ be the number of positive integers in $\lambda$. 

More generally, we define a partition function of Dellac configurations 
with general boundaries $\lambda$ and $\mu$.
as follows.
The weight of a Dellac configuration is $q^{\mathrm{inv}(C)}$.
\begin{defn}
We define a partition function of 
Dellac configurations with boundaries $\lambda$ and $\mu$
by 
\begin{align*}
\mathfrak{C}^{n}(\lambda,\mu):=\sum_{C}q^{\mathrm{inv}(C)},
\end{align*}
where the sum is all over Dellac configurations with boundaries
$\lambda$ and $\mu$.

\begin{remark}
When $\lambda$ and $\mu$ are the staircases, the partition functions 
coincide with the Poincar\'e polynomials for the degenerate flag 
varieties studied in \cite{Fei11}.
\end{remark}

When $\mu=\delta_{n-1}$, we abbreviate $\mathfrak{C}^{n}(\lambda,\mu)$
as $\mathfrak{C}^{n}(\lambda)$.
\end{defn}

By definition, $\mathfrak{C}^{n}(\lambda,\mu)$ is a polynomial of $q$ and we have 
\begin{align*}
C^{n}(\lambda,\mu)=\mathfrak{C}^{n}(\lambda,\mu)|_{q=1}.
\end{align*}

\begin{example}
\label{ex:qGDC}
When $n=3$, $\lambda=(2,1), (2), (1,1), (1)$ or $\emptyset$, and $\mu=\delta_{2}$, we have
\begin{align*}
\mathfrak{C}^{3}((2,1))&=1+2q+3q^2+q^3, \\
\mathfrak{C}^{3}((2))&=1+2q+3q^2+2q^3+q^4, \\
\mathfrak{C}^{3}((1,1))&=1+2q+4q^2+3q^3+2q^4, \\
\mathfrak{C}^{3}((1))&=1+2q+4q^2+4q^3+3q^4+q^5, \\
\mathfrak{C}^{3}(\emptyset)&=1+2q+4q^2+4q^3+4q^4+2q^5+q^6.
\end{align*}
In general, the polynomial $\mathfrak{C}^{n}(\lambda)$ is not palindromic.
The coefficient of the top degree in $\mathfrak{C}^{n}(\lambda)$ may not
be $1$.
\end{example}

The top degree in $\mathfrak{C}^{n}(\lambda)$ can be computed by the 
following lemma.
\begin{lemma}
Let $M^{n}(\lambda)$ be the maximal number of inversions of Dellac 
configurations of size $n$ with boundaries $\lambda$ and $\delta_{n-1}$.
Then, the value $M^{n}(\lambda)$ is given by 
\begin{align}
\label{eqn:deg}
M^{n}(\lambda)=n(n-1)-\sum_{i}\lambda_{i}.
\end{align}
\end{lemma}
\begin{proof}
We prove the lemma by induction.
When $\lambda=\delta_{n-1}$, the maximal number of inversions 
is equal to $n(n-1)/2$, which is $M^{n}(\delta_{n-1})$.

Any partition $\lambda$ inside of $\delta_{n-1}$ can be obtained 
by deleting boxes one by one by starting from the staircase $\delta_{n-1}$.
We assume that Eqn. (\ref{eqn:deg}) is true for $\lambda$. 
Suppose that we delete a box in the $i$-th row and the $j$-th column from $\lambda$ 
to obtain a new partition $\lambda'$.
We call this box $(i,j)$-box.
Then, we consider a configuration with a dot in the $(i,j)$-box.
By definition, we have one dot in the $i$-th row and two dots in the $j$-th 
column.
Let $C_1(\lambda)$ be a configuration with a boundary $\lambda$ 
such that it has a maximal number of inversions.
We construct $C_{1}(\lambda')$ from $C_1(\lambda)$ as follows.
We put a dot $d$ in the $(i,j)$-box in $C_{1}(\lambda)$, 
and delete two dots which is left to $d$ and below $d$ in $C_{1}(\lambda)$.
Suppose deleted dots are in the $(i,j_1)$-box and the $(i_1,j)$-box in $C_1(\lambda)$.
Then, we add a dot in the $(i_1,j_1)$-box.
In total, we add two dots and delete two dots, and the new configuration is 
a Dellac configuration with the boundary $\lambda'$. 
Further, the number of inversions is 
increased by one compared to that of $C_1(\lambda)$, and 
it is easy to see that this configuration has the maximal number of inversions. 
Thus, we have Eqn. (\ref{eqn:deg}).
\end{proof}

When $l:=l(\lambda)\le n+1$, we regard $\lambda$ as a partition of size $n+1$ 
by appending $n+1-l(\lambda)$ $0$'s to $\lambda$, 
{\it i.e.}, we define a new $\lambda$ by $\lambda=(\lambda_1,\ldots,\lambda_{l},0^{n+1-l})$.

\begin{prop}
\label{prop:Crr}
The value $\mathfrak{C}^{n}_{\lambda}$ satisfies the following recurrence relations.
\begin{enumerate}
\item $\lambda_1=n-1$. 
Then, we have 
\begin{align}
\label{eqn:GDCb1}
\mathfrak{C}^{n}(\lambda)=\sum_{2\le i\le n+1}q^{i-2}\mathfrak{C}^{n-1}(\lambda^{i}),
\end{align}
where $\lambda^{i}:=\lambda\setminus\{\lambda_1,\lambda_{i}\}$ for $2\le i\le n+1$. 
\item $\lambda_1\le n-2$.
Then, we have 
\begin{align}
\label{eqn:GDCb2}
\mathfrak{C}^{n}(\lambda)=\sum_{1\le i<j\le n+1}q^{i+j-3}\mathfrak{C}^{n-1}(\lambda^{i,j}),
\end{align}
where $\lambda^{i,j}=\lambda\setminus\{\lambda_i,\lambda_j\}$.
\end{enumerate}
\end{prop}

\begin{proof}
For case (1), we have only one place to put a dot in the top row.
This dot is in the rightmost column. In this column, we put one more 
dot $d$ somewhere in the $i$-th row with $2\le i\le n+1$.
There are $i-2$ dots which is above and left to the dot $d$, we have 
the weight $q^{i-2}$ coming from inversions relevant to $d$.
By taking the sum for such configurations together with the weight $q^{n-2}$, 
we have Eqn. (\ref{eqn:GDCb1}).

Similarly for case (2), we put two dots in the rightmost column in any 
two rows. 
Suppose that the two dots are in the $i$-th and $j$-th rows with 
$1\le i<j\le n+1$.
The number of inversions relevant to these two dots are $i+j-3$.
Thus we have the weight $q^{i+j-3}$ for these configurations.
By taking the sum for such configurations together with the weight, 
we have Eqn. (\ref{eqn:GDCb2}).
\end{proof}

\begin{example}
We calculate $C^{n}(\lambda)$ by setting $q=1$ in Proposition \ref{prop:Crr}.
Set $n=4$, $\lambda=(3,1)$ and $\mu=(3,2,1)$.
Then we have 
\begin{align*}
C^{4}(\{3,1\})&=3 C^{3}(\{1\})+ C^{3}(\emptyset), \\
&=3 (3 C^{2}(\{1\})+3C^{2}(\emptyset))+C^{3}(\emptyset), \\
&=3(3*2+3*3)+18, \\
&=63.
\end{align*}
\end{example}

Let $\delta(n-1:i)$ be the Young diagram obtained from $\delta_{n-1}$
by deleting the integer $i$ for $1\le i\le n-1$. In case of $i=0$,
we define $\delta(n-1:0)=\delta_{n-1}$. 
The first few values of $C^{n}(\delta(n-1:i))$ are shown as follows.
The values $C^{n}(\delta(n-1:i))$ for $0\le i\le n-1$ are placed 
from left to right.
\begin{center}
\begin{tabular}{c|ccccc}
n$\backslash$ i & 0 & 1 & 2 & 3 & 4  \\ \hline
1 & 1 \\
2 & 2 & 3 \\ 
3 & 7 & 9 & 15 \\
4 & 38 & 45 & 63 & 111 \\
5 & 295 & 333 & 423 & 621 & 1131 \\
\end{tabular}
\end{center}

By applying Proposition \ref{prop:Crr} with $q=1$, we have 
\begin{align*}
C^{n}(\delta(n-1))=2C^{n-1}(\delta(n-2))
+\sum_{i=1}^{n-2}C^{n-1}(\delta(n-2:i)),
\end{align*}
which is the total number of Dellac configurations 
with boundaries $\lambda=\mu=\delta_{n-1}$.

Below, we derive several recurrence relations for 
the total number of Dellac configurations of 	size $n$ 
with boundaries $\lambda$ and $\delta_{n-1}$.
Then, we prove Theorem \ref{thrm:Crr} which is a recurrence relations 
with three terms.

\begin{lemma}
\label{lemma:Crr2}
Let $\lambda:=(\lambda_1,\ldots,\lambda_l)$ be a partition with $\lambda_1\le n-3$.
We have

\begin{align}
\label{eqn:Crr2}
\mathfrak{C}^{n}((n-1,n-2)\oplus\lambda)
+q^{n}\mathfrak{C}^{n-1}((n-2)\oplus\lambda)
=[3]_{q}\mathfrak{C}^{n-1}(\lambda).
\end{align}
\end{lemma}
\begin{proof}
We expand each term in Eqn. (\ref{eqn:Crr2}) as a sum of $\mathfrak{C}^{n-2}(\lambda)$ 
for some $\mu\subseteq\lambda$ by using Proposition \ref{prop:Crr}.
Then, we compare the coefficients of $\lambda$ in both sides of Eqn. (\ref{eqn:Crr2}).

The coefficients of $\mathfrak{C}_{n-2}(\lambda)$ in the left hand side of Eqn. (\ref{eqn:Crr2})
are
\begin{align}
\begin{aligned}
\left(\sum_{l+1\le i\le n-1}q^{i}\right)\left(\sum_{l\le i\le n-2}q^{i}\right)
+\sum_{l\le i<j\le n-1}q^{i+j-1}
=q^{2l+1}[n-1-l]^{2}_q+q^{2l}\genfrac{[}{]}{0pt}{}{n-l}{2}_{q},
\end{aligned}
\end{align}
from $\mathfrak{C}^{n}((n-1,n-2)\oplus\lambda)$ and $q^{l+n}[n-1-l]_q$
from $q^{n}\mathfrak{C}^{n-1}((n-2)\oplus\lambda)$.
The sum of these two contributions is equal to
\begin{align}
q^{2l}[3]_q\genfrac{[}{]}{0pt}{}{n-l}{2}_{q}.
\end{align}
Similarly, we have a coefficient 
\begin{align}
[3]_q\left(\sum_{l\le i<j\le n-1}q^{i+j-1}\right)=q^{2l}[3]_q\genfrac{[}{]}{0pt}{}{n-l}{2}_{q}
\end{align}
of $[3]_q\mathfrak{C}^{n-2}(\delta(n-3))$ in the right hand side.
Therefore, we have 
the same coefficient.

Similarly, the coefficients of $\mathfrak{C}^{n-2}(\lambda\setminus\{\lambda_{i}\})$ for $1\le i\le l$ 
in both sides of Eqn. (\ref{eqn:Crr2}) are $q^{l-i-2}[3]_q[n-l]_q$.

We denote $\lambda^{i,j}:=\lambda\setminus\{\lambda_i,\lambda_j\}$ for $1\le i<j\le l$.
Then, the coefficients of $\mathfrak{C}^{n-2}(\lambda^{i,j})$ for $1\le i<j\le l$ in both sides of 
Eqn. (\ref{eqn:Crr2}) are $[3]_qq^{i+j-3}$.

In all cases, the coefficients of $C^{n-2}(\mu)$ in both sides of Eqn. (\ref{eqn:Crr2})
are the same, and this completes the proof. 
\end{proof}

\begin{cor}
We have 
\begin{align*}
\mathfrak{C}^{n}(\delta_{n-1})+q^{n}\mathfrak{C}^{n-1}(\delta_{n-2})=[3]_q\mathfrak{C}^{n-1}(\delta_{n-3}).
\end{align*}
\end{cor}
\begin{proof}
We set $\lambda=\delta_{n-3}$ in Lemma \ref{lemma:Crr2}.
\end{proof}

Let $\lambda:=(\lambda_1,\ldots,\lambda_l)$ be a partition of length $l$ and 
$\lambda\oplus1:=(\lambda_1,\ldots,\lambda_l,1)$.
We denote $\mathbf{1}_{l}:=(1^{l})$ and 
$\lambda-\mathbf{1}_{l}:=(\lambda_1-1,\ldots,\lambda_{l}-1)$.
\begin{lemma}
\label{lemma:Crr1}
We have 
\begin{align}
\mathfrak{C}^{n}(\lambda)=\mathfrak{C}^{n}(\lambda\oplus1)+
q^{\alpha_1}\mathfrak{C}^{n-1}(\lambda-\mathbf{1}_{l}),
\end{align}
where $\alpha_1:=2n-2-l(\lambda)$.
\end{lemma}
\begin{proof}
Dellac configurations with boundaries $\lambda$ and $\mu=\delta(n-1)$ are divided into 
two classes. The first class is the set of configurations such that 
the box $(l+1,1)$ has no dots, and the second class contains a dots in
the box $(l+1,1)$ where $l$ is the length of $\lambda$.
The number of configurations of the first class is $\mathfrak{C}^{n}(\lambda\oplus1)$.
For the number of configurations of the second class, note that 
the first column contains two dots in the boxes $(l+1,1)$ and $(2n,1)$ 
since we have $\mu=\delta(n-1)$.
A configuration $C_{1}$ in the second class is bijective to a configuration 
obtained from $C_{1}$ by deleting first column, and $l+1$-th and $2n$-th
rows. By this deletion, we have the boundaries $\lambda-\mathbf{1}_{l}$ 
and $\mu\setminus\{n-1\}=\delta(n-2)$.
The total number of such configurations is 
$\mathfrak{C}^{n-1}(\lambda-\mathbf{1}_{l})$.
Note that the dot in $(l+1,1)$-box has $\alpha_1:=2n-2-l(\lambda)$ inversions. 
So, we have a weight $q^{\alpha_1}$ for $\mathfrak{C}^{n-1}(\lambda-\mathbf{1}_{l})$.
By taking the sum of total number of two classes, this completes the proof.
\end{proof}

\begin{lemma}
\label{lemma:Crr4d}
Let $\lambda$ and $\nu$ be a partition such that 
$\min\{\lambda_{i}: 1\le i\le l(\lambda) \}\ge m+1$ and 
$\max\{\nu_{i}: 1\le i\le l(\nu)\}\le m-1$.
Then, we have 
\begin{align}
\label{eqn:Crr41d}
\begin{aligned}
(1+q^2)\mathfrak{C}^{n}(\lambda\oplus(m)\oplus\nu)
=&\mathfrak{C}^{n}(\lambda\oplus(m-1)\oplus\nu)+
q^{2}\mathfrak{C}^{n}(\lambda\oplus(m+1)\oplus\nu) \\
&+q^{\alpha}\mathfrak{C}^{n-1}((\lambda-\mathbf{1})\oplus\nu),
\end{aligned}
\end{align}
where 
\begin{align}
\alpha=2n-l(\lambda)-m.
\end{align}
\end{lemma}
\begin{proof}
We prove the lemma by induction.
When $n=3$ and $m=1$, we have 
\begin{align*}
(1+q^2)\mathfrak{C}^{3}((1))=\mathfrak{C}^{3}(\emptyset)+q^2\mathfrak{C}^{3}((2))
+q^5\mathfrak{C}^{2}(\emptyset),
\end{align*}
by a straightforward calculation.

We assume that Eqn. (\ref{eqn:Crr41d}) is true for up to $n-1$ and arbitrary 
$\lambda$ and $\mu$.
From Proposition \ref{prop:Crr}, we have
\begin{align}
\label{eqn:Crr42}
\begin{aligned}
\mathfrak{C}^{n}(\lambda\oplus(m)\oplus\nu)
=&\sum_{1\le i<j\le l(\lambda)}q^{i+j-3}\mathfrak{C}^{n-1}(\lambda^{i,j}\oplus(m)\oplus\nu)
+\sum_{i}q^{i+l(\lambda)-2}\mathfrak{C}^{n-1}(\lambda^{i}\oplus\nu) \\
&+\sum_{1\le i\le l(\lambda)}\sum_{1\le j\le l(\nu)}q^{i+j+l(\lambda)-2}
\mathfrak{C}^{n-1}(\lambda^{i}\oplus(m)\oplus\nu^{j}) \\
&+q^{l(\lambda)+l(\nu)}[c_1]_q\sum_{1\le i\le l(\lambda)}q^{i-1}
\mathfrak{C}^{n-1}(\lambda^{i}\oplus(m)\oplus\nu) \\
&+q^{2l(\lambda)}\sum_{1\le i\le l(\nu)}q^{i-1}\mathfrak{C}^{n-1}(\lambda\oplus\nu^{i})
+q^{l(\lambda)+l(\nu)}[c_1]_{q}\mathfrak{C}^{n-1}(\lambda\oplus\nu) \\
&+q^{2(l(\lambda)+1)}\sum_{1\le i<j\le l(\nu)}q^{i+j-3}\mathfrak{C}^{n-1}(\lambda\oplus(m)\oplus\nu^{i,j})\\
&+q^{2l(\lambda)+l(\nu)+1}[c_1]_q\sum_{1\le i\le l(\nu)}q^{i-1}\mathfrak{C}^{n-1}(\lambda\oplus(m)\oplus\nu^{i}) \\
&+q^{2(l(\lambda)+l(\nu)+1)}c_2\mathfrak{C}^{n-1}(\lambda\oplus(m)\oplus\nu),
\end{aligned}
\end{align}
where 
\begin{align*}
c_1=n-l(\lambda)-l(\nu), \quad c_2=\genfrac{[}{]}{0pt}{}{c_1}{2}_{q}.
\end{align*}
By substituting Eqn. (\ref{eqn:Crr42}) into Eqn. (\ref{eqn:Crr41d}) and 
by induction on $n$, the difference between the left hand side and first two 
terms of the right hand side of Eqn. (\ref{eqn:Crr41d}) is equal to 
$q^{\alpha}$ times 
\begin{multline*}
\sum_{1\le i<j\le l(\lambda)}q^{i+j-3}\mathfrak{C}^{n-2}((\lambda^{i,j}-\mathbf{1})\oplus\nu)
+\sum_{1\le i\le l(\lambda)}\sum_{1\le j\le l(\nu)}q^{i+j+l(\lambda)-2}
\mathfrak{C}^{n-2}((\lambda^{i}-\mathbf{1})\oplus\nu^{j}) \\
\quad+[c_1]_q\sum_{1\le i\le l(\lambda)}q^{i+l(\lambda)+l(\nu)-2}
\mathfrak{C}^{n-2}((\lambda^{i}-\mathbf{1})\oplus\nu) \\
\quad+\sum_{1\le i<j\le l(\nu)}q^{i+j+2l(\lambda)-4}\mathfrak{C}^{n-2}((\lambda-\mathbf{1})\oplus\nu^{i,j}) \\
\quad+[c_1]_q\sum_{i}q^{i+l(\lambda)-1}\mathfrak{C}^{n-2}((\lambda-\mathbf{1})\oplus\nu^{i})
+c_2 q^{2l(\lambda)+2l(\nu)}\mathfrak{C}^{n-2}((\lambda-\mathbf{1})\oplus\nu),
\end{multline*}
which is equal to $q^{\alpha}\mathfrak{C}^{n-1}((\lambda-\mathbf{1})\oplus\nu)$.
\end{proof}

\begin{remark}
The value $\alpha$ in Eqn. (\ref{eqn:Crr41}) can be expressed by 
\begin{align*}
\alpha=2+M^{n}(\lambda\oplus(m)\oplus\nu)-M^{n-1}((\lambda-\mathbf{1}_{l(\lambda)})\oplus\nu).
\end{align*}
Thus, the value $\alpha$ detects the difference of the degrees of the top terms in partition functions 
$\mathfrak{C}^{n}(\lambda\oplus(m)\oplus\nu)$ and $\mathfrak{C}^{n-1}((\lambda-\mathbf{1})\oplus\nu)$.
\end{remark}

In Lemma \ref{lemma:Crr4}, we have a recurrence relation 
for four terms. 
The choice of the coefficients is not uniquely fixed.
In the following lemma, we give another recurrence relation 
for the same four terms.
\begin{lemma}
\label{lemma:Crr4}
Let $\lambda$ and $\nu$ be a Young diagram such that 
$\min\{\lambda_{i}: 1\le i\le l(\lambda) \}\ge m+1$ and 
$\max\{\nu_{i}: 1\le i\le l(\nu)\}\le m-1$.
Then, we have 
\begin{align}
\label{eqn:Crr41}
\begin{aligned}
(1+q)\mathfrak{C}^{n}(\lambda\oplus(m)\oplus\nu)
=&\mathfrak{C}^{n}(\lambda\oplus(m-1)\oplus\nu)+
q\mathfrak{C}^{n}(\lambda\oplus(m+1)\oplus\nu) \\
&+q^{\alpha_0}\mathfrak{C}^{n-1}((\lambda-\mathbf{1})\oplus\nu),
\end{aligned}
\end{align}
where 
\begin{align}
\alpha_0=2n-m-2-l(\lambda).
\end{align}
\end{lemma}
\begin{proof}
The proof is essentially same as the proof of Lemma \ref{lemma:Crr4d}.
We expand both sides of Eqn. (\ref{eqn:Crr41}) by Proposition \ref{prop:Crr}, 
and show that they are equal by induction.
\end{proof}

\begin{lemma}
\label{lemma:Crr5}
We have 
\begin{align}
\label{eqn:Crr3}
\begin{aligned}
&\sum_{1\le i<j\le l(\lambda)}q^{i+j-3}\mathfrak{C}^{n-2}(\lambda^{i,j}\oplus\nu)
+\sum_{1\le i\le l(\lambda)}\sum_{1\le j\le l(\nu)}q^{i+j+l(\lambda)-3}\mathfrak{C}^{n-2}(\lambda^{i}\oplus\nu^{j}) \\
&\qquad+[c_1]_q \sum_{1\le i\le l(\lambda)}q^{i+l(\lambda)+l(\nu)-2}\mathfrak{C}^{n-2}(\lambda^{i}\oplus\nu) 
+ [c_1]_q\sum_{1\le i\le l(\nu)}q^{i+2l(\lambda)+l(\nu)-2}\mathfrak{C}^{n-2}(\lambda\oplus\nu^{i}) \\ 
&\qquad+\sum_{1\le i<j\le l(\nu)}q^{i+j+2l(\lambda)-3}\mathfrak{C}^{n-2}(\lambda\oplus\nu^{i,j})
+ c_2 q^{2(l(\lambda)+l(\nu))} \mathfrak{C}^{n-2}(\lambda\oplus\nu) \\
&\qquad=\mathfrak{C}^{n-1}(\lambda\oplus\nu)-q^{n-2}\mathfrak{C}^{n-1}((n-2)\oplus\lambda\oplus\nu),
\end{aligned}
\end{align}
where
\begin{align*}
c_{1}=n-1-l(\lambda)-l(\nu), \quad
c_{2}=\genfrac{[}{]}{0pt}{}{c_1}{2}_{q}.
\end{align*}
\end{lemma}
\begin{proof}
One can easily obtain Eqn. (\ref{eqn:Crr3}) by applying Proposition \ref{prop:Crr}
to the right hand side of Eqn. (\ref{eqn:Crr3}).
\end{proof}

We will establish a recurrence relation among three terms by 
using Lemmas \ref{lemma:Crr4} and \ref{lemma:Crr5}.
\begin{theorem}
\label{thrm:Crr}
Let $\lambda:=(\lambda_1,\ldots,\lambda_{l})$ be a Young diagram such that 
$\lambda_{i}\ge m+1$ for all $1\le i\le l$.
Similarly, let $\nu:=(\nu_{1},\ldots,\nu_{l'})$ be a Young diagram such that  
$\nu_{1}\le m-1$.
We define 
\begin{align*}
\kappa_{0}&=(\lambda,m,m,\nu), \\
\kappa_{1}&=(\lambda,m+1,m-1,\nu), \\
\kappa_{2}&=(\lambda-\mathbf{1}_{l},\nu).
\end{align*}
Then, we have
\begin{align}
\label{eqn:rrCn}
\mathfrak{C}^{n}(\kappa_0)=\mathfrak{C}^{n}(\kappa_1)+q^{\beta}\mathfrak{C}^{n-1}(\kappa_2),
\end{align}
where 
\begin{align*}
\beta=2(n-m-1)-l(\lambda).
\end{align*}
\end{theorem}
\begin{proof}
We prove the theorem by induction.
When $n=2$, we have 
\begin{align}
\mathfrak{C}^{2}((1,1))=\mathfrak{C}^{1}(\emptyset)=1,
\end{align}
by a straightforward computation using Proposition \ref{prop:Crr}.

We assume that the recurrence relation (\ref{eqn:rrCn}) is true 
up to $n-1$.
We expand $\mathfrak{C}^{n}(\kappa_{i})$ with $i=\{0,1\}$ by Proposition \ref{prop:Crr}.
More precisely, we have 
\begin{align*}
\begin{aligned}
\mathfrak{C}^{n}(\kappa_{0})
=&\sum_{1\le i<j\le l(\lambda)}q^{i+j-3}\mathfrak{C}^{n-1}(\lambda^{i,j}\oplus(m,m)\oplus\nu)
+[2]_q\sum_{1\le i\le l(\lambda)}q^{i+l(\lambda)-2}\mathfrak{C}^{n-1}(\lambda^{i}\oplus(m)\oplus\nu)  \\
&+\sum_{1\le i\le l(\lambda)}\sum_{1\le j\le l(\nu)}q^{i+j+l(\lambda)-1}
\mathfrak{C}^{n-1}(\lambda^{i}\oplus(m,m)\oplus\nu^{j}) \\
&+[c_1]_q\sum_{1\le i\le l(\lambda)}q^{i+l(\lambda)+l(\nu)}\mathfrak{C}^{n-1}(\lambda^{i}\oplus(m,m)\oplus\nu) \\
&+q^{2l(\lambda)}\mathfrak{C}^{n-1}(\lambda\oplus\nu)
+[2]_q\sum_{i}q^{i+l(\lambda)+l(\nu)}\mathfrak{C}^{n-1}(\lambda\oplus(m)\oplus\nu^{i})  \\
&+[2]_q[c_1]_qq^{2l(\lambda)+l(\nu)+1}\mathfrak{C}^{n-1}(\lambda\oplus(m)\oplus\nu) \\
&+\sum_{1\le i<j\le l(\nu)}q^{i+j+2l(\lambda)+1}\mathfrak{C}^{n-1}(\lambda\oplus(m,m)\oplus\nu^{i,j})  \\
&+[c_1]_q\sum_{1\le i\le l(\nu)}q^{i+2l(\lambda)+2+l(\nu)}\mathfrak{C}^{n-1}(\lambda\oplus(m,m)\oplus\nu^{i}) \\
&+c_2 q^{2l(\lambda)+2l(\nu)+4}\mathfrak{C}^{n-1}(\lambda\oplus(m,m)\oplus\nu),
\end{aligned}
\end{align*}
and 
\begin{align*}
\begin{aligned}
\mathfrak{C}^{n}(\kappa_{1})
=&\sum_{1\le i<j\le l(\lambda)}q^{i+j-3}\mathfrak{C}^{n-1}(\lambda^{i,j}\oplus(m+1,m-1)\oplus\nu) \\
&+\sum_{1\le i\le l(\lambda)}q^{i+l(\lambda)-1}\mathfrak{C}^{n-1}(\lambda^{i}\oplus(m+1)\oplus\nu) \\
&+ \sum_{1\le i\le l(\lambda)}q^{i+l(\lambda)-2}\mathfrak{C}^{n-1}(\lambda^{i}\oplus(m-1)\oplus\nu) \\
&+\sum_{1\le i\le l(\lambda)}\sum_{1\le j\le l(\nu)}q^{i+j+l(\lambda)-1}
\mathfrak{C}^{n-1}(\lambda^{i}\oplus(m+1,m-1)\oplus\nu^{j}) \\
&+[c_1]_q\sum_{1\le i\le l(\lambda)}q^{i+l(\lambda)+l(\nu)}
\mathfrak{C}^{n-1}(\lambda^{i}\oplus(m+1,m-1)\oplus\nu) \\
&+q^{2l(\lambda)}\mathfrak{C}^{n-1}(\lambda\oplus\nu) 
+\sum_{1\le i\le l(\nu)}q^{i+2l(\lambda)+1}\mathfrak{C}^{n-1}(\lambda\oplus(m+1)\oplus\nu^{i}) \\
&+\sum_{1\le i\le l(\nu)}q^{i+2l(\lambda)}\mathfrak{C}^{n-1}(\lambda\oplus(m-1)\oplus\nu^{i})  \\
&+[c_1]_q q^{2l(\lambda)+l(\nu)+2}\mathfrak{C}^{n-1}(\lambda\oplus(m+1)\oplus\nu)
+[c_1]_qq^{2l(\lambda)+l(\nu)+1}\mathfrak{C}^{n-1}(\lambda\oplus(m-1)\oplus\nu) \\
&+\sum_{1\le i<j\le l(\nu)}q^{i+j+2l(\lambda)+1}\mathfrak{C}^{n-1}(\lambda\oplus(m+1,m-1)\oplus\nu^{i,j})\\
&+[c_1]_q\sum_{1\le i\le l(\nu)}q^{i+2l(\lambda)+l(\nu)+2}\mathfrak{C}^{n-1}(\lambda\oplus(m+1,m-1)\oplus\nu^{i}) \\
&+c_2q^{2l(\lambda)+2l(\lambda)+4}\mathfrak{C}^{n-1}(\lambda\oplus(m+1,m-1)\oplus\nu),
\end{aligned}
\end{align*}
where $c_1=n-1-l(\lambda)-l(\nu)$ and $c_2=\genfrac{[}{]}{0pt}{}{c_1}{2}_q$.
From these expressions, Lemmas \ref{lemma:Crr4} and \ref{lemma:Crr5}, 
and
\begin{align*}
\mathfrak{C}^{n-1}((n-2)\oplus\lambda'\oplus\nu)
&=\sum_{1\le i\le l(\lambda')}q^{i-1}\mathfrak{C}^{n-1}(\lambda'^{i}\oplus\nu) 
+\sum_{1\le i\le l(\nu)}q^{i+l(\lambda')-1}\mathfrak{C}^{n-2}(\lambda'\oplus\nu^{i}) \\
&\qquad+[c_1]_qq^{l(\lambda)+l(\nu)}\mathfrak{C}^{n-2}(\lambda'\oplus\nu),
\end{align*}
where $\lambda'=\lambda-\mathbf{1}$, 
we have 
\begin{align*}
\mathfrak{C}^{n}(\kappa_0)-\mathfrak{C}^{n}(\kappa_1)
&=q^{\beta}\mathfrak{C}^{n-1}((\lambda-\mathbf{1})\oplus\nu) \\
&=q^{\beta}\mathfrak{C}^{n-1}(\kappa_{2}).
\end{align*}
This completes the proof.
\end{proof}

We have three kinds of recurrence relations for $C^{n}(\lambda)$, Proposition \ref{prop:Crr},
Lemma \ref{lemma:Crr2} and Lemma \ref{lemma:Crr4}.
To compute the total number of Dellac configurations by Proposition \ref{prop:Crr}, 
we first compute the values $C^{n}(\lambda)$, where $\lambda=\delta_{n-1}\setminus\{i\}$ 
with $1\le i\le n-1$.
Since such $\lambda$'s play a central role in the computation,
it is natural to ask whether the value $C^{n}(\mu)$ is expressed in terms of 
such $C^{n}(\lambda)$'s.
Here, $\mu$ is a partition such that each integer appears at most twice 
in $\mu$. 
The following theorem implies that there exists an expression of $C^{n}(\mu)$
in terms of $C^{n}(\lambda)$'s.

\begin{theorem}
\label{thrm:lc}
Let $\lambda$ be a partition of size $l$ such that each integer appears at most twice in $\lambda$.
Then, the value $C^{n}_{\lambda}$ can be expressed as a linear combination of $C^{n'}_{\mu}$ 
such that $n'-l(\mu)=1$ or $2$.
\end{theorem}
\begin{proof}
We prove the statement by induction on $n$ and $l$.
For $n=1$ and $n=2$, the statement is true since 
\begin{align*}
C^{2}(\emptyset)
=\genfrac{}{}{1pt}{}{1}{3}\left( C^{3}((2,1))+C^{2}((1))\right).
\end{align*} 
by Lemma \ref{lemma:Crr2}.

We assume that the statement is true up to all $n$ and $l$ such that 
$n-l=m$. 
In case of $n$, we have two cases: 1) $\lambda_1<n-1$, and  
2) $\lambda_1=n-1$.
We define $m(n,l):=n-l$. 

\paragraph{Case 1)} 
We apply Lemma \ref{lemma:Crr2} on $C^{n}(\lambda)$, which results 
in a linear combination of $C^{n+1}((n,n-1)\oplus\lambda)$ and 
$C^{n}((n-1)\oplus\lambda)$.
Note that the length of $(n,n-1)\oplus\lambda$ is $l+2$ and 
the length of $(n)\oplus\lambda$ is $l+1$.
In both cases, $m(n,l)$ is decreased by one after the application of 
Lemma \ref{lemma:Crr2}. 
By the induction assumption, these two terms can be expressed in terms of 
$C^{n'}(\mu)$ such that $n'-l(\mu)=1$ or $2$.
Thus, $C^{n}(\lambda)$ can be expressed in terms of $C^{n'}(\mu)$'s.

\paragraph{Case 2)}
Since $\lambda_1=n-1$, we apply Lemma \ref{lemma:Crr4} and to $\lambda_1$ in 
$C^{n}(\lambda)$.
By this operation, we may obtain a partition $\mu$ 
such that an integer appears twice in $\mu$.
Otherwise, we apply Lemma \ref{lemma:Crr2} as Case 1) and the statement 
is true by induction.
We consider the case where an integer appears twice in $\mu$.
We apply Theorem \ref{thrm:Crr} to $\mu$ until we have no repeated integers 
in $\mu$.
By construction, a partition $\mu$ satisfies $\mu_1\le n-2$.
We have two cases: a) $n'=n$ and b) $n'<n$.
\paragraph{Case 2a)}
We apply Lemma \ref{lemma:Crr2} to $C^{n'}(\mu)$ and obtain a partition
such that $m(n',l(\mu))$ is one less than $m(n,l(\lambda))$.
We continue these processes until $m(n',l(\mu))$ becomes $1$ or $2$.

\paragraph{Case 2b)} 
By induction assumption, $C^{n'}(\mu)$ is a linear combination of 
$C^{n''}(\mu')$ with $m(n'',l(\mu'))=0$ or $1$, since we have $n'<n$.

These observations complete the proof.
\end{proof}

\begin{remark}
For a general partition $\lambda$, the expression of $C^{n}(\lambda)$ in 
terms of $C^{n'}(\mu)$ may contain a term with $n'>n$.
For example, $C^{5}_{1}$ is expressed as 
\begin{align*}
C^{5}_1=\genfrac{}{}{1pt}{}{1}{30}C^{7}_{65431}
+\genfrac{}{}{1pt}{}{1}{6}C^{6}_{5431}
+\genfrac{}{}{1pt}{}{7}{30}C^{5}_{431}
+\genfrac{}{}{1pt}{}{1}{10}C^{4}_{31}.
\end{align*} 
\end{remark}

One can state a stronger statement than Theorem \ref{thrm:lc} in case of 
$\mu=\delta_{n-1}\setminus\{i\}$ for $1\le i\le n-1$.
\begin{cor}
Let $\lambda:=\delta_{n-1}\setminus\{i\}$ be a partition with $1\le i\le n-1$.
Then, the value $C^{n}_{\lambda}$ can be expressed as a linear 
combination of $C^{n'}_{\mu_{n'}}$ 
such that $\mu_{n'}=\delta_{n-1}\setminus\{i+1\}$ for $n'=n$
and $\mu_{n'}=\delta_{n-1}\setminus\{n-1,\ldots,n'\}$ for $n'\in[i+1,n-1]$.
\end{cor}
\begin{proof}
Given a pair of $(n,\lambda)$, we define $m(n,l(\lambda)):=n-l$.
When $\lambda=\delta_{n-1}\setminus\{i\}$, we have $m(n,l(\lambda))=2$.
We apply Lemma \ref{lemma:Crr4} to $C^{n}(\lambda)$, and obtain 
a linear combinations of $C^{n'}(\mu)$ such that 
$m(n',l(\mu))=2$ when $n'=n$ and $m(n',l(\mu))=2$ or $3$ when $n'<n$.
We apply Lemma \ref{lemma:Crr2} to $C^{n'}(\mu)$'s with $m(n',l(\mu))=3$ and 
$n'<n$.
Note that application of Lemma \ref{lemma:Crr2} decreases $m(n',l(\mu))$ by one
and may increase $n'$ by one.
Therefore, an expansion of $C^{n}(\lambda)$ 
in terms of $C^{n'}(\mu)$ is a linear combination
of $C^{n'}(\mu)$ such that $n'\le n$ and $m(n',l(\mu))=1$ or $2$.
A partition $\mu$ is written as $\mu=\delta_{n-1}\setminus\{n-1,\ldots,n'\}$ 
for some $n'\in[i+1,n-1]$ since we append $(n-1,n-2)$ or $(n-2)$ to $\mu$
when we apply Lemma \ref{lemma:Crr2} to $C^{n'}(\mu)$.
This completes the proof.
\end{proof}

\begin{example}
We compute $C^{8}_{765421}$ as follows.
We have 
\begin{align}
\label{ex:r1}
C^{8}_{765421}=\genfrac{}{}{1pt}{}{1}{2}
\left(C^{8}_{765321}+C^{7}_{6521}+C^{7}_{6421}+C^{7}_{5421}+C^{7}_{65421}\right),
\end{align}
by Lemma \ref{lemma:Crr4} and successive applications of Theorem \ref{thrm:Crr}.
We calculate the right hand side of Eqn. (\ref{ex:r1}) as
\begin{align}
C^{7}_{6521}&=
\genfrac{}{}{1pt}{}{1}{2}C^{7}_{6421}
+\genfrac{}{}{1pt}{}{1}{6}C^{7}_{6521}
+\genfrac{}{}{1pt}{}{2}{3}C^{6}_{521},\\
C^{7}_{6421}&=
\genfrac{}{}{1pt}{}{1}{6}\left( 
C^{8}_{765421}+2C^{7}_{65421}+C^{6}_{5421}
\right), \\
C^{7}_{5421}&=\genfrac{}{}{1pt}{}{1}{3}\left(
C^{8}_{765421}+C^{7}_{65421}
\right), \\
C^{6}_{521}&=
\genfrac{}{}{1pt}{}{1}{6}\left(
C^{7}_{65421}+2C^{6}_{5421}+C^{5}_{421}
\right).
\end{align}
Substituting these expressions into Eqn. (\ref{ex:r1}),
we obtain 
\begin{align*}
C^{8}_{765421}
=\genfrac{}{}{1pt}{}{5}{7}C^{8}_{765321}
+\genfrac{}{}{1pt}{}{10}{7}C^{7}_{65421}
+\genfrac{}{}{1pt}{}{8}{21}C^{6}_{5421}
+\genfrac{}{}{1pt}{}{2}{21}C^{5}_{421}.
\end{align*}
\end{example}

\bibliographystyle{amsplainhyper} 
\bibliography{biblio}

\end{document}